\documentclass[11pt, a4paper]{amsart}

\usepackage[margin=2.5cm]{geometry}

\usepackage{epsf} 

\usepackage{listings} 
\usepackage{algorithmic,algorithm}
\usepackage{multirow}
\usepackage{enumerate}
\usepackage{booktabs}

\usepackage[]{graphicx}
\usepackage{amscd}
\usepackage[pdftex]{color} % black,white,red,green,blue,cyan,magen ta,yellow
\usepackage[pdftex,colorlinks]{hyperref}
\usepackage{lscape}
\usepackage{indentfirst}
\usepackage{latexsym}
\usepackage{amsmath, amsfonts, amssymb,mathrsfs}
\usepackage{verbatim}

\usepackage{tikz}
\usepackage{pgfplots, pgfplotstable}
\usepgfplotslibrary{groupplots}

\usepackage{bm}

\usepackage[extra]{tipa}

\newcommand{\vertiii}[1]{{\left\vert\kern-0.25ex\left\vert\kern-0.25ex\left\vert #1 
    \right\vert\kern-0.25ex\right\vert\kern-0.25ex\right\vert}}

\newtheorem{theorem}{Theorem}
\newtheorem{lemma}[theorem]{Lemma}

\newtheorem{remark}{Remark}

\newtheorem{definition}{Definition}

\DeclareMathOperator{\ddiv}{div}

\def\g{{\bm g}}

\def\ooy{{\overline{\overline{\bm y}}}}
\def\oox{{\overline{\overline{\bm x}}}}
\def\ooX{{\overline{\overline{\bm X}}}}
\def\ooA{{\overline{\overline{\mathcal{A}}}}}
\def\hu{{\hat{\bm u}}}
\def\hv{{\hat{\bm v}}}
\def\hp{{\hat{p}}}
\def\hq{{\hat{q}}}

\def\sig{{\bm \sigma}}

\def\bU{{\boldsymbol U}}

\def\bu{{\boldsymbol u}}
\def\bv{{\boldsymbol v}}
\renewcommand{\div}{\operatorname{div}}

\newcommand{\ds}{\mathop{\mathrm{d} s}}
\newcommand{\dx}{\mathop{\mathrm{d} x}}

\renewcommand{\div}{\operatorname{div}}
\newcommand{\T}{{\operatorname{T}}}
\def\g{\gamma}

\def\eps{\boldsymbol \epsilon}
\def\sig{\boldsymbol \sigma}

\def\bU{\boldsymbol U}

\def\bW{\boldsymbol W}

\def\bf{\boldsymbol f}

\def\bu{\boldsymbol u}
\def\bv{\boldsymbol v}
\def\bw{\boldsymbol w}

\def\bx{\boldsymbol x}
\def\by{\boldsymbol y}
\def\bz{\boldsymbol z}

\def\bn{\boldsymbol n}

\def\divv{\text{div}}

\def\tmu{\tilde{\mu}}
\def\tlambda{\tilde{\lambda}}
\usepackage{siunitx}
\newcommand{\numeoc}[1]{\num[round-precision=1,round-mode=places, scientific-notation=false]{#1}}

\title[Uniformly well-posed HDG/HM discretizations for Biot's model] {Uniformly well-posed hybridized discontinuous
    Galerkin/hybrid mixed discretizations for Biot's consolidation
    model}

  \thanks{% Submitted. % to the editors DATE.
    PL and JS acknowledge the funding by the Austrian Science Fund
    (FWF) through the research programm ``Taming complexity in partial
    differential systems'' (F65) - project ``Automated discretization
    in multiphysics'' (P10).}

  \author[J. Kraus]{Johannes Kraus}
\address{Faculty of Mathematics, University Duisburg-Essen, Germany}
\email{johannes.kraus@uni-due.de}

\author[P.~L.~Lederer]{Philip L. Lederer}
\address{Institute for Analysis and Scientific Computing, TU Wien, Austria}
\email{philip.lederer@tuwien.ac.at}

  \author[M. Lymbery]{Maria Lymbery}
\address{Faculty of Mathematics, University Duisburg-Essen, Germany}
\email{maria.lymbery@uni-due.de}

\author[J. Sch\"oberl]{Joachim Sch\"oberl}
\address{Institute for Analysis and Scientific Computing, TU Wien, Austria}
\email{joachim.schoeberl@tuwien.ac.at}

\begin{document}

\begin{abstract}
We consider the quasi-static Biot's consolidation model in a
three-field formulation with the three unknown physical quantities
of interest being the displacement $\bm u$ of the solid matrix, the
seepage velocity $\bm v$ of the fluid and the pore pressure $p$.  As
conservation of fluid mass is a leading physical principle in
poromechanics, we preserve this property using an
$\bm H({\rm div})$-conforming ansatz for $\bm u$ and $\bm v$
together with an appropriate pressure space. This results in Stokes
and Darcy stability and exact, that is, pointwise mass conservation
of the discrete model.

The proposed discretization technique combines a hybridized
discontinuous Galerkin method for the elasticity subproblem with a
mixed method for the flow subproblem, also handled by hybridization.
The latter allows for a static condensation step to eliminate the
seepage velocity from the system while preserving mass
conservation. The system to be solved finally only contains degrees
of freedom related to $\bm u$ and $\bm p$ resulting from the
hybridization process and thus provides, especially for higher-order
approximations, a very cost-efficient family of physics-oriented
space discretizations for poroelasticity problems.

We present the construction of the discrete model, theoretical
results related to its uniform well-posedness along with optimal
error estimates and parameter-robust preconditioners as a key tool
for developing uniformly convergent iterative solvers. Finally, the
cost-efficiency of the proposed approach is illustrated in a series
of numerical tests for three-dimensional test cases.
\end{abstract}

\keywords{
  Biot's consolidation model, strongly mass-conserving high-order
  discretizations, parameter-robust LBB stability, norm-equivalent
  preconditioners, hybrid discontinuous Galerkin methods, hybrid mixed
  methods}

\maketitle

%\tableofcontents

%\newpage

\section{Introduction}\label{sec:intro}

Poroelastic models describing the mechanical behaviour of fluid saturated
porous media find a wide range of applications in many different fields of science, medicine and engineering. 
%such as 
%earth sciences... . 
The theory of poroelasticity was initially conceived by Maurice Anthony Biot who, in the period between 1935 and 1962, 
see e.g.~\cite{Biot1941general, Biot1955theory}, 
proposed a soil consolidation model 
%an extension of soil consolidation models developed 
to calculate the settlement of structures placed on 
fluid-saturated porous soils.

Recently, interest in Biot's consolidation equations has been revived due to their newly discovered applications in medicine, 
see e.g.~\cite{Sebaa2008application} and~\cite{Guo2018subject},
where they have been studied in the context of human cancellous bone samples
and risk factors associated with the early stages of Alzheimer's disease, respectively. 
Their numerical solution has consequently been a subject of active research. 
One major challenge is that the parameters involved in Biot's model can vary over many orders of magnitude 
%in different practical applications 
and, therefore, it is vital that not only the variational formulation of the problem is stable
but also that the iterative solution method is uniformly convergent over the whole range of admissible model parameter values. 
%Other approaches, such as discretization by finite differences and finite volume methods 
%are also possible and have been considered 
%in~\cite{Axelsson2012stable, gaspar2003finite, gaspar2006staggered, nordbotten2016stable,Kumar2020conservative}.

A rigorous stability and convergence analysis for finite element (FE)
approximations of the two-field formulation of Biot's equations where
the velocity field has been eliminated from the unknowns has first
been presented in~\cite{Murad1992improved,Murad1994onstability}.  The
derived a priori error estimates are valid for both semidiscrete and
fully discrete formulations, where the backward Euler method is used
for time-discretization and inf-sup stable finite elements are used
for space discretization.

% In~\cite{Yi2014convergence}, a discretization for the four-field
% formulation of Biot's consolidation problem in two space dimensions
% has been proposed. The method uses the (effective) stress tensor,
% the fluid flux, the displacement and the pore pressure as unknowns
% and couples two standard mixed finite element methods, one for the
% flow and one for the mechanics subproblem, the latter of which is
% based on the Hellinger-Reissner formulation.  The error analysis of
% this coupled mixed method has been complemented
% in~\cite{Lee2016robust} by estimates in $L^{\infty}$ norm in time
% and $L_2$ norm in space that are robust with respect to the Lam\'e
% parameters and do not require strict positivity of the constrained
% storage coefficient.  Introducing the skew-symmetric part of the
% gradient of the displacement as an additional unknown, which acts as
% a Lagrange multiplier for enforcing the symmetry of the stress
% tensor, the method studied herein works with weakly symmetric stress
% approximations, which can be a preferable alternative especially in
% view of computational costs and ease of implementation,
% cf. hybridized mixed methods~\cite{arnold1985mixed, MR2051067}.

Other recent developements in discretizing Biot-type models are
related to the stabilization of conforming
methods~\cite{rodrigo2018newstabilized}, stable finite volume
methods~\cite{nordbotten2016stable}, discretizations for
total-pressure-based formulations~\cite{lee2017parameter,
  oyarzua2016locking}, including conservative discontinuous finite
volume and mixed schemes~\cite{Kumar2020conservative}, enriched
Galerkin methods~\cite{girault2020aposteriori, lee2018enriched},
space-time finite element approximations~\cite{bause2017spacetime},
and methods for two-phase flow and non-linear extensions of the Biot
problem~\cite{lee2018enriched, radu2018arobust}, to mention only but a
few. Finally, and, nevertheless, important in the context of the
present research, are the extensions of abovementioned discretization
techniques to multicompartmental (multiple network) poroelasticity
problems presented in~\cite{lee2019spacetime, hong2020parameter}.

The subject of the study in this paper is the standard three-field
formulation of Biot's model in which the unknown fields are the
displacement, seepage velocity and fluid pressure. Discretizations
based on three-field-formulation have originally been proposed
in~\cite{phillips2007coupling_a, phillips2007coupling_b} where
con\-tinuous-in-time and discrete-in-time error estimates have been
proved. This approach has also been extended to discountinuous
Galerkin approximations of the displacement field
in~\cite{phillips2008coupling} and other nonconforming approximations,
e.g., using modified rotated bilinear elements~\cite{Yi2013coupling},
or Crouzeix-Raviart elements for the displacements in
\cite{hu2017anonconforming}. More recently, in~\cite{HongKraus2018}, a
family of strongly mass conserving discretizations based on the
$\bm H$(div)-conforming dicontinuous Galerkin (DG) discretization of the
displacement field has been suggested and its parameter-robust
stability and near best approximation properties 
%with respect to allmodel and discretization parameters 
proven. Time-dependent error
estimates for the same family of discretizations have been proved
in~\cite{kanschat2018afinite}.
%and for a hybridized two-field version in
%\cite{FU2019237}.
Note that these approaches are based on the inf-sup stability of the
corresponding Stokes discretization scheme which were originally
stated in~\cite{CockburnEtAl2005locally, CockburnEtAl2002local,
  cockburn2007note} and the Brinkman problem
\cite{stenberg2011,MR2860674}.

Hybridization techniques have been applied to discretizations of
Biot's model in the recent works \cite{FU2019237} and
\cite{Niu2020}. Whereas in \cite{FU2019237} the authors introduced a
hybridized $\bm H$(div)-conforming DG method for the two-field
formulation, the work \cite{Niu2020} starts from a lowest-order
conforming stabilized discretization of the three-field formulation
and uses hybridization for the flow subsystem as it was first
presented in the \cite{MR813687}.

The aim of the present work is the construction, analysis and
numerical testing of a new family of higher-order mass-conserving
hybridized/hybrid mixed FE discretizations for the three-field
formulation of Biot's model.  The main focus lies on a well-posedness
analysis in properly scaled norms resulting in estimates with
constants that are independent of any problem parameters.  As a
consequence, we obtain norm-equivalent preconditioners and optimal
near best approximation estimates. 

%utilizing a mixed-hybrid formulation for the flow problem and static
%condensation techniques. This approach is demonstrated to be highly
%efficient and especially for higher-order approximations.

The paper is structured as follows. In Section~\ref{sec:probl} the
governing equations are stated and the three-field formulation of
Biot's model is discussed.  Its semi-discretization in time by the
implicit Euler method along with a proper rescaling of the parameters
results in a static boundary value problem and is presented in
Section~\ref{sec:discr}.  The latter then is discretized in space by a
new family of hybridized discontinuous Galerkin/hybrid mixed methods
while addressing the advantages of this approach.  The main
theoretical results follow in Section~\ref{sec:preco} where the
uniform boundedness and the parameter-robust inf-sup stability of the
underlying bilinear form are proven to be independent of all model and
discretization parameters. Furthermore, the corresponding
parameter-robust preconditioners and error estimates are provided.  In
Section~\ref{numres} the theoretical results of this paper are
complemented by a series of numerical tests assessing the
approximation quality and cost efficiency of these preconditioners for
the proposed family of higher-order hybridized discontinuous
Galerkin/hybrid mixed discretizations.
%%%%%%%%%%%%%%%%

%%%%%%%%%%%%%%%%

\section{Problem formulation}\label{sec:probl}

\subsection{Governing equations}

We consider a porous medium, which is linearly elastic, homogeneous,
isotropic and saturated by an incompressible Newtonian fluid.
Then Biot's consolidation model, see~\cite{Terzaghi1925erdbaumechanic,Biot1941general}, for a bounded
Lipschitz domain $\Omega\in\mathbb{R}^d$, $d\in\{2,3\}$,
\begin{subequations}\label{eq::twofield}
  \begin{alignat}{2} 
-\divv (2 \tmu \epsilon(\bu)) - \tlambda \nabla \divv \bu + \alpha \nabla p
&= \tilde{\bf},  \quad &&\mbox{in } \Omega \times (0,T), \\
\frac{\partial}{\partial t}(S_0 p + \alpha \divv \bu) - \divv (K \nabla p)
&= \tilde{g},  \quad &&\mbox{in } \Omega \times (0,T),
\end{alignat}
\end{subequations}
relates the deformation $\bu$ and the fluid pressure $p$ 
% of
%an isotropric elastic porous medium saturated with an incompressibe
%fluid 
for a given body force density $\tilde{\bf}$ and mass source or sink $\tilde{g}$.  
For convenience, we assume a scalar conductivity
coefficient $K$. In this work, we use bold symbols to denote vector- or
tensor-valued quantities, e.g., $\eps(\bu) := \frac{1}{2}(\nabla \bu + (\nabla \bu)^T)$
denoting the symmetric gradient. Further,
$\tlambda$ and $\tmu$ are the Lam\'e parameters, $\alpha$ is the Biot-Willis parameter and $S_0$
the constrained specific storage coefficient. 

The three-field~\cite{phillips2007coupling_a,phillips2008coupling}
formulation is based on the primary variables $(\bu, \bm{w}, p)$,
i.e.,
\begin{alignat*}{2}
- \divv {\sig} &= \tilde{\bf}, \quad &&\mbox{in } \Omega \times (0,T), \\
K^{-1} \bm{w} + \nabla p &= {\boldsymbol 0}, \quad &&\mbox{in } \Omega \times (0,T), \\
\frac{\partial}{\partial t}(S_0 p + \alpha \, \divv \, \bu) + \divv \, \bm{w}
&=\tilde{g},  \quad &&\mbox{in } \Omega \times (0,T),
%\A {\sig} - \left(\eps(\bu) - \frac{\alpha}{d\tlambda + 2\tmu} p\boldsymbol I \right) &= 0 , \quad &&\mbox{in } \Omega \times (0,T),
\end{alignat*}
where $\bm{w}$ denotes the seepage velocity,
$\tilde {\sig} := 2 \tmu \eps(\bu) + \tlambda \divv (\bu)
\boldsymbol I$ is the total stress and
${\sig} = \tilde{\sig} - \alpha p \boldsymbol I$ the effective
stress.
%Here $\A$ denotes the fourth order compliance tensor of linear elasticity defined by 
%\begin{align*}
%\A \ta &:=\frac{1}{2 \tmu} \left( \ta - \frac{\tlambda}{d \tlambda + 2 \tmu} {\text tr} (\ta) \boldsymbol I \right).
%\end{align*}
%In this work we consider the three-field formulation
%obtained by eliminating the stress variable $\bm \sigma$ using the
%above constitutive relations.
If not mentioned otherwise, we assume homogeneous Dirichlet
boundary conditions for the displacement $\bu$ and homogeneous
Neumann conditions for the pressure~$p$. 
In this context, let $\bm{H}^1_0(\Omega), \bm{H}_0(\ddiv, \Omega)$ denote the standard
vector-valued Sobolev spaces where the subscript $0$ refers to
homogeneous essential boundary conditions. Further, let $L^2_0(\Omega)$
denote the space of square integrable functions with zero mean
value. Following the standard procedure, one 
%multiply each equation with suitable test functions and integrate by parts to 
derives the weak formulation: Find
$(\bm{u},\bm{w}, p)\in \bm{H}^1_0(\Omega) \times \bm{H}_0(\ddiv,
\Omega) \times L^2_0(\Omega)$ such that
\begin{subequations}\label{eqn:weak-3f}
\begin{alignat}{2}
\tilde{a}(\bm{u},\bm{v}) - (\alpha \,p,\ddiv \bm{v}) &= (\tilde{\bm f},\bm{v}),
     \quad &&\forall \bm{v}\in \bm{H}^1_0(\Omega), \label{eqn:weak-3f-1}\\
(K^{-1}\bm{w},\bm{z}) - (p,\ddiv \bm{z}) &= 0, \quad && \forall  \bm{z}\in \bm{H}_0(\ddiv, \Omega),\label{eqn:weak-3f-2}\\
- (\alpha\ddiv \partial_t \bm{u},q) - (\ddiv \bm{w},q) - (S_0 \partial_t  p, q) &= -(\tilde{g},q), \quad && \forall q \in L^2_0(\Omega),\label{eqn:weak-3f-3}
\end{alignat}
\end{subequations}
where
\begin{align}
\tilde{a}(\bm{u},\bm{v}) &:= 2\tmu\int_{\Omega}\varepsilon(\bm{u}):\varepsilon(\bm{v}) \dx + \tlambda\int_{\Omega} \ddiv\bm{u}\ddiv\bm{v} \dx.
\label{u_form_2f}
%\\
%\tilde{a}_p(p,q) &= \int_{\Omega} K\nabla p\cdot \nabla q + \int_{\Omega} S_0 \partial_t  p \, q .
%\label{p_form_2f}
\end{align}
Finally, system \eqref{eqn:weak-3f} is completed with suitable
initial conditions $\bm{u}(\cdot, 0) = \bm{u}_0(\cdot)$ and
$p(\cdot, 0) = p_0(\cdot)$.

\section{Hybridized discontinuous Galerkin/hybrid mixed discretizations of the Biot problem}\label{sec:discr}

\subsection{Strongly mass-conserving discretization of the Biot problem}

The starting point for this subsection is a family of strongly
mass-conserving discretizations of the three-field formulation of the
quasi-static Biot model based on a discontinuous Galerkin (DG)
formulation for the mechanics subproblem, as proposed
in~\cite{HongKraus2018}. After time discretization by the implicit
Euler scheme, the method for the arising static problem in each time
step can be expressed as follows:

Find the time-step functions 
$(\bm{u}^k, \bm{w}^k, p^k) := (\bm{u}(\bx,t_k),\bm{w}(\bx,t_k), p(\bx,t_k))\in \bm{H}^1_0(\Omega) \times \bm{H}_0(\ddiv, \Omega) \times L^2_0(\Omega)$ 
which solve the following system of equations 
\begin{subequations}\label{eqn:weak-3f-iE}
\begin{alignat}{2}
\tilde{a}(\bm{u}^k,\bm{v}) - (\alpha \,p^k,\ddiv \bm{v}) &= (\tilde{\bm f}^k,\bm{v}),
     \quad &&\forall \bm{v}\in \bm{H}^1_0(\Omega), \label{eqn:weak-3f-1-iE}\\
(K^{-1}\bm{w}^k,\bm{z}) - (p^k,\ddiv \bm{z}) &= 0, \quad &&
\forall  \bm{z}\in \bm{H}_0(\ddiv, \Omega),\label{eqn:weak-3f-2-iE}\\
- (\alpha\ddiv (\bm{u}^k-\bm{u}^{k-1}),q)  - \tau(\ddiv \bm{w}^k,q) 
  - (S_0 (p^k-p^{k-1}), q) & = -\tau(\tilde{g}^k,q), 
 \quad &&\forall q \in L^2_0(\Omega),\label{eqn:weak-3f-3-iE}
\end{alignat}
\end{subequations}
where $\tau$ is the time-step parameter and $\tilde{\bm f}^k = \tilde{\bm f}(\bx,t_k)$, $\tilde{ g}^k = \tilde{ g}(\bx,t_k)$.

For the space discretization, consider a shape-regular triangulation $\mathcal{T}_h$ whose set of facets are denoted by~$\mathcal{F}_h$.
We introduce the following finite element spaces
\begin{align*}
\bm U_h&:=\{\bv \in \bm H_0(\divv, \Omega):\bv|_T \in \bU(T),~T \in \mathcal{T}_h\}, \\
\bm W_{h}&:=\{\bz \in \bm H_0(\divv ,\Omega):\bz|_T \in \bm W(T),~T \in \mathcal{T}_h\},
\\
P_{h}&:=\{q \in L_0^2(\Omega):q|_T \in P(T),~T \in \mathcal{T}_h\}.
\end{align*}
The local spaces $\bm U(T)$, $\bm W(T)$, $P(T)$ are either 
$\text{BDM}_{\ell}(T)$, $\text{RT}_{\ell-1}(T)$,
$\text{P}_{\ell-1}(T)$ or by $\text{BDFM}_{\ell}(T)$,
$\text{RT}_{\ell-1}(T)$, $\text{P}_{\ell-1}(T)$
where $\text{BD(F)M}_{\ell}(T)$, $\text{RT}_{\ell-1}(T)$, and
$\text{P}_{\ell-1}(T)$ denote the local
Brezzi-Douglas-(Fortin-)Marini space of order $\ell$, the
Raviart-Thomas space of order $\ell-1$, and full polynomials of degree
$\ell-1$, respectively. A definition of these local spaces can be found, for example, in~\cite{Boffi2013mixed}.

We present the definitions of some trace operators next.  Let
$F = \partial T_1 \cap \partial T_2$ be a common facet of two adjacent elements 
$T_1,T_2 \in \mathcal{T}_h$, and let $\bm n_1, \bm n_2$ be
the corresponding outward pointing unit normal vectors.
%to be unit normal vectors to $F$ directed to the
%exterior of $T_1$ and $T_2$, respectively.
%
%....
%For any edge (or face) 
%$F \in \mathcal{E}_h^{I}$ 
For any interior facet $F \not \subset \partial \Omega$ and
element-wise smooth and scalar-valued function $q$, vector-valued function $\bm v$ and
tensor-valued function $\bm \tau$, 
their averages and jumps on the facet $F$ are defined by 
\begin{equation*}
  \{\bm v\} =\frac{1}{2}(\bm v_1\cdot \bm n_1-\bm v_2\cdot \bm n_2), \quad \{\bm \tau\}=\frac{1}{2}(\bm
  \tau_1 \bm n_1-\bm \tau_2 \bm n_2), \quad
  [q]=q_1-q_2,\quad [\bm v]=\bm v_1-\bm v_2,
\end{equation*}
where the subscript $i$, $i = 1,2$, with the functions $q$, $\bm v$ and $\bm \tau$ refers to their evaluation on $T_i \cap F$. 
%\begin{equation*}
%\begin{split}
%\{\bm v\} &=\frac{1}{2}(\bm v|_{\partial T_1\cap F}\cdot \bm n_1-\bm
%v|_{\partial T_2\cap F}\cdot \bm n_2), \quad 
%\{\bm \tau\}=\frac{1}{2}(\bm \tau|_{\partial T_1\cap
%F} \bm n_1-\bm \tau|_{\partial T_2\cap F} \bm n_2),
%\end{split}
%\end{equation*}
%and the jumps as
%\begin{equation*}
%[q]=q|_{\partial T_1\cap F}-q|_{\partial T_2\cap F},\quad
%[\bm v]=\bm v|_{\partial T_1\cap F}-\bm v|_{\partial T_2\cap F}.
%\quad
 %\Lbrack\bm v\Rbrack=\bm v|_{\partial T_1\cap F}\odot \bm n_1+\bm v|_{\partial T_2\cap F}\odot \bm n_2,
%\end{equation*}
%where $\bm v \odot \bm n=\frac{1}{2}(\bm{v} \bm{n}^T+\bm n \bm v^T)$ is the 
%symmetric part of the tensor product of $\bm v$ and $\bm n$.
For any boundary facet $F \subset \partial \Omega$, 
%$e \in  \mathcal{E}_h^{B}$ then the above quantities are defined as
these quantities are given as
\[
\{\bm v\}=\bm v |_{F}\cdot \bm n,\quad
\{\bm \tau\}=\bm \tau|_{F}\bm n, \quad
[q]=q|_{F}, \quad [\bm v]=\bm v|_{F}.
%, \quad  \Lbrack\bm v\Rbrack=\bm v|_{F}\odot \bm n.
\]
With these definitions at hand, the formulation of the method is as
follows: Find
$(\bu_h, \bw_h, p_h ) \in \bU_h \times \bm W_h \times P_h$, such that
%for any
%$(\bv_h, \bz_h, q_h) \in \bm U_h \times \bm W_h \times  P_h$ it holds that
%
\begin{subequations}\label{Cons-disc-3-field}
\begin{alignat}{2}
 a_{h}(\bu_h,\bv_h)- 
 (p_{h}, \divv \bv_h)
 &=(\bf, \bv_h),\quad &&\forall \bv_h \in \bm U_h, \label{Cons-disc-3-field_a} \\
  (R^{-1}\bw_{h},\bz_{h}) {-} (p_{h},\divv \bz_{h}) &= 0, \quad &&\forall \bz_h \in \bm W_h, \label{Cons-disc-3-field_b} \\
 	 -\ (\divv \bu_h,q_{h}) - (\divv\bw_{h},q_{h}) 
 	 -(S p_{h},q_{h}) 
 	 &= (g,q_{h}), \quad &&\forall q_h \in P_h.  \label{Cons-disc-3-field_c}
\end{alignat}
\end{subequations}
This system has been derived by dividing system~\eqref{eqn:weak-3f-iE} by $2 \tmu$ and, additionally, 
equation~\eqref{eqn:weak-3f-2-iE} by the time step size $\tau$ and furthermore by applying the 
substitutions $\bm u_h=\alpha \bm u_h^k$, $\bm w_h = \tau \bm w_h^k$, $p_h = \alpha^2 p_h^k$. 
%with $t_k = t_{k-1} + \tau$ denoting a given time moment.
The right-hand sides in~\eqref{Cons-disc-3-field} are 
$\bf= \alpha\tilde{\bf}(\bx,t_{k})/2\tmu$ and $g = (\tau \tilde{g}(\bx,t_{k})-\alpha \divv(\bu_h(\bx,t_{k-1})) -
c_{0}p(\bx,t_{k-1}))/2\tmu$,

%that is,
%$\bq_h=(\bq_{1,h}^T,\bq_{2,h}^T,\ldots,\bq_{n,h}^T)^T, \bz_h=(\bz_{1,h}^T,\bz_{2,h}^T,\ldots,\bz_{n,h}^T)^T
% \in \bQ_h=\bQ_{1,h} \times \bQ_{2,h} \times \ldots \times \bQ_{n,h}$ and
% $\bp_h=(p_{1,h},p_{2,h},\ldots,p_{n,h})^T,\br_h=(r_{1,h},r_{2,h},\ldots,r_{n,h})^T \in \bP_h=P_{1,h} \times P_{1,h} \times \ldots \times P_{n,h}$.
%
$$
a_h(\bu_h,\bv_h) := a_h^{\text{DG}}(\bu_h,\bv_h) + \lambda \int_{\Omega}  \divv \bu_h\, \divv \bv_h \dx
$$
and 
\begin{equation}\label{scaled_param}
\lambda:=\frac{\tlambda}{2\tmu},\quad R:=\frac{2\tmu\tau}{\alpha^2} K  > 0, \quad S:=\frac{2 \tmu S_0}{\alpha^2}.
\end{equation}
Note that the discrete bilinear form $a_h(\cdot,\cdot)$ is obtained from scaling the bilinear form in~\eqref{u_form_2f} by 
$1 / 2 \tmu$.
%$$a(\bu,\bv)=  \int_{\Omega}\varepsilon(\bm{u}):\varepsilon(\bm{v}) \dx +\lambda \int_{\Omega}  \divv \bu\, \divv \bv \dx.
%$$
We denote the tangential component of any vector field on a facet by its symbol with a 
subscript~$t$. 
Then the symmetric interior penalty Galerkin (SIPG) bilinear form $a_h^{\text{DG}}(\cdot,\cdot)$ is 
defined as
\begin{align}
a_h^{\text{DG}}(\bu,\bv)  :=& \sum _{T \in \mathcal{T}_h} \int_K \eps(\bu) : \eps(\bv) \dx 
-\sum_{F \in \mathcal{F}_h} \int_F \{\eps(\bu)\} \cdot [\bv_t] \ds
\nonumber \\
&- \sum _{F \in \mathcal{F}_h} \int_F \{\eps(\bv)\} \cdot [\bu_t] \ds 
+ \sum _{F \in \mathcal{F}_h} \eta \ell^2 h_F^{-1} \int_F [ \bu_t] \cdot [\bv_t] \ds
\label{a_h_DG}
\end{align}
with a sufficiently large stabilization parameter $\eta$ independent
of all model parameters, i.e., $\lambda, \, R, \, S$, and discretization parameters $h$ and $\tau$.  
%The subscript $t$ when used with a vector variable indicates that we take the tangential component of its vector field.
%whereas the parameter $\tau$ 
%denotes the time step size and 
Note that in this paper the constants in all parameter robust estimates are independent of
%constants parameter robustness is also understood as
%independence of 
model and discretization parameters.

%\subsection{Hybridized DG method for the mechanics subproblem}
\subsection{Hybridized DG method} \label{sec::HDG} When dealing with
Stokes-type problems, $\bm H(\text{div})$-conforming discretizations
possess several advantages over $H^1$-conforming
discretizations. This is mainly due to the fact that they allow for a
suitable approximation of the incompressibility constraint which
results in favorable properties such as pointwise divergence-free
solutions and pressure robustness, see,
e.g.,~\cite{CockburnEtAl2002local,CockburnEtAl2005locally}.  However,
the incorporation of (tangential) continuity in standard DG schemes
leads to a significantly increased number of (globally) coupled
degrees of freedom~(dof). To overcome this, in
\emph{hybridized DG methods}, one decouples element unknowns by
introducing additional unknowns on the facets through
which (tangential) continuity is imposed weakly, see,
e.g.,~\cite{CockburnEtAl2009unified,LehrenfeldSchoeberl2016high}.

In the context of an $\bm H(\text{div})$-conforming hybridized DG discretization, one introduces an additional
space 
%$\widehat{\bU}_h$ 
$$
\widehat{\bU}_h:=\{ \hat{\bu} \in \bm{L}^2(\mathcal{F}_h): \hat{\bu}|_F \in \bm P_{\ell}(F) \textrm{ and }  
\hat{\bu}|_F \cdot \bn=0, ~ F\in \mathcal{F}_h; ~ \hat{\bu} = \bm 0 \text{ on }\partial{\Omega} \}
$$
for the approximation of the tangential trace of the displacement
field $\bu$. Here, $\bm{L}^2(\mathcal{F}_h)$ denotes the space of vector-valued
square integrable functions on the skeleton $\mathcal{F}_h$ and
$\bm P_{\ell}(F)$ the vector-valued polynomial space of order
$\ell$ on each facet $F \in \mathcal{F}_h$. We replace the
bilinear form $a_h^{\text{DG}}(\cdot,\cdot)$ defined in~\eqref{a_h_DG}
by $a_h^{\text{HDG}}(\cdot,\cdot)$ given by
\begin{align}
a_h^{\text{HDG}}((\bu,\hat{\bu}),(\bv,\hat{\bv})):=&
\sum_{T\in\mathcal{T}_h}\left[ \int_{T} \eps(\bu) : \eps(\bv) \, \dx + \int_{\partial T} \eps(\bu) \bn \cdot (\hat{\bv}-\bv)_t\, \ds \right. \nonumber \\
&+ \left. 
\int_{\partial T} \eps(\bv) \bn \cdot (\hat{\bu}-\bu)_t\, \ds +  \eta \ell^2 h^{-1}\int_{\partial T} (\hat{\bu}-\bu)_t \cdot (\hat{\bv}-\bv)_t \ds \right],
\label{a_h_HDG}
%%\right. \nonumber 
%\\
%&+&\left. \eta \ell^2 h^{-1}\int_{\partial T} (\hat{\bu}-\bu)_t \cdot (\hat{\by}-\by)_t \,ds \right], \label{a_h_HDG}
\end{align}
%\begin{eqnarray}
%a_h^{\text{HDG}}((\bu,\hat{\bu}),(\by,\hat{\by}))&:=&
%\sum_{T\in\mathcal{T}_h}\left[ \int_{T} \eps(\bu) : \eps(\by) \, d x \right. \nonumber \\
%&+& \left. \int_{\partial T} \eps(\bu) \bn \cdot (\hat{\by}-\by)_t\, \ds
%+\int_{\partial T} \eps(\by) \bn \cdot (\hat{\bu}-\bu)_t\, \ds \right. \nonumber \\
%&+&\left. \eta \ell^2 h^{-1}\int_{\partial T} (\hat{\bu}-\bu)_t \cdot (\hat{\by}-\by)_t \,ds \right], \label{a_h_HDG}
%\end{eqnarray}
where $(\bu, \hat{\bu})$,
$(\bv, \hat{\bv})\in \overline{\bU}_h := \bU_h \times
\widehat{\bU}_h$.  Our approach for exactly divergence-free hybridized
discontinuous Galerkin methods will be based on~\cite{LehrenfeldSchoeberl2016high} as well as its improvements
presented in~\cite{LedererEtAl2019hybrid,LedererEtAl2018hybrid}.  The
resulting method for the Biot problem now reads as: Find
$((\bu_h, \hat{\bu}_h), \bw_h, p_h ) \in \overline{\bm U}_h \times \bm
W_h \times P_h$, such that
%for any
%$((\bv_h, \hat{\bv}_h), \bz_h, q_h) \in \overline{\bm U}_h \times \bm W_h \times  P_h$ it holds that
%
\begin{subequations}\label{Cons-disc-3-field-HDG}
\begin{alignat}{2}
  a_{h}((\bu_h, \hat{\bu}_h),(\bv_h, \hat{\bv}_h))- (p_{h}, \divv
  \bv_h)
  &=(\bf, \bv_h), \quad &&\forall (\bv_h, \hat{\bv}_h) \in \overline{\bm U}_h, \label{Cons-disc-3-field-HDG_a} \\
  (R^{-1}\bw_{h},\bz_{h}) {-} (p_{h},\divv \bz_{h}) &= 0,  \quad &&\forall \bz_h \in \bm W_h,  \label{Cons-disc-3-field-HDG_b} \\
  -\ (\divv \bu_h,q_{h}) - (\divv\bw_{h},q_{h}) -(S p_{h},q_{h}) &=
  (g,q_{h}), \quad && \forall q_h \in
  P_h, \label{Cons-disc-3-field-HDG_c}
\end{alignat}
\end{subequations}
where 
\begin{equation}\label{a_h}
a_{h}((\bu_h, \hat{\bu}_h),(\bv_h, \hat{\bv}_h)) := a_h^{\text{HDG}}((\bu_h,\hat{\bu}_h),(\bv_h,\hat{\bv}_h)) 
+\lambda \int_{\Omega}  \divv \bu_h \divv \bv_h \dx
\end{equation}
and $a_h^{\text{HDG}}((\cdot,\cdot),(\cdot,\cdot))$ is defined in~\eqref{a_h_HDG}.

%\subsection{Hybrid mixed method for the flow subproblem}
\subsection{A family of hybridized DG/hybrid mixed methods} \label{sec::hybridizeddarcy} 
In this subsection, we enrich the hybridization idea by additionally introducing a hybrid mixed formulation for the flow 
subproblem. 
%The parameter-robust
%norm-equivalent preconditioners discussed
%in~\cite{HongKraus2018,Hong2019conservativeMPET} were based on the
%stability analysis of the three field formulation of the Biot-problem,
%see equation \eqref{Cons-disc-3-field}. 
While the stability analysis presented in~\cite{HongKraus2018,Hong2019conservativeMPET} 
uses properly scaled $\bm H(\text{div})$ and an $L^2$ norms for the flow subproblem, we hybridize 
the latter one in the present work. This approach has the advantage that when solving the full saddle-point problem 
with some preconditioned iterative method, one needs to invert a div-grad type operator instead of a grad-div operator in order to 
apply the preconditioner which is easier and more cost-efficient in general. 
% 
%For the analysis of the Darcy flow subproblem we use standard scaled Sobolev norms, thus an
%$\bm H(\text{div})$-norm and an $L^2$-norm for the velocity and the
%pressure, respectively. A disadvantage of this approach is that a
%preconditioned system requires a grad-div type solver for the
%velocity. To overcome this problem, we will hybridize the flow sub
%problem which allows to choose different norms for the analysis
%resulting in standard (DG-like) div-grad like preconditioners for the
%final system. 
Note that the solution of the hybridized system is the
same as that of the non-hybridized one.
%
%In the next section we will then extend the stability analysis originally presented in~\cite{HongKraus2018,Hong2019conservativeMPET} 
%to the resulting new family of strongly mass-conserving discretizations. 

The additional hybridization step can be expressed as follows. First, one enforces the normal continuity of the velocity by a Lagrange 
multiplier. To this end, we introduce the following finite element spaces

\begin{align*}
  \bm W_{h}^-&:=\{\bz \in \bm L^2(\Omega):\bz|_T \in \bm W(T),~T \in \mathcal{T}_h \},\quad
               %;~ \bz \cdot
%\bm n=0~\hbox{on}~\partial \Omega\},\\
  \widehat P_h :=\prod_{F\in \mathcal{F}_h}\text{P}_{l-1}(F),\quad
                 %\left\{p_F \in\prod_{F\in \mathcal{F}_h}\text{P}_{l-1}(F): \; 
                 %p_F = 0 \text{ on }\partial{\Omega}  \right\}, \\
  \overline P_h := P_h \times \widehat P_h,                 
\end{align*}
%\plnote{Randbedingungen muessen wir noch anpassen, je nachdem was wir dann machen. Dann auch ein kommentar dass die natuerlichen und essentiellen RB sich wieder umdrehen.}
where $\bm W(T)$ can be chosen in the same way as before. Here, the space $\bm W_{h}^-$ 
is simply a discontinuous version of the space $\bm W_{h}$. Further, note that $\widehat P_h$ is chosen as 
the normal trace space of $\bm W_h$, e.g., in the case $\bm W(T) = RT_0$ (thus $l-1 = 0$) the normal traces on each facet 
are constant and so correspondingly we also choose $\widehat P_h$ to be defined as facet wise constants. 
Based on these spaces, we next define for all $\bm w_h \in \bm W_{h}^-$ and $(p_h, \hat p_h) \in \overline P_h$ the bilinear form
\begin{align}\label{form_b}
  b((p_h, \hat p_h),\bm w_h) = \sum_{T \in \mathcal{T}_h} \left(\int_T \div \bm w_h p_h \dx - \int_{\partial T} \bm w_h \cdot \bm n \hat p_h \ds\right).
\end{align}
This bilinear form can be interpreted as a distributional version of the
inner product ``$(\div \bw_h, p_h)$''
since functions in $\bm W_{h}^-$ are not normal continuous. 
Therefore, variational problem~\eqref{Cons-disc-3-field-HDG}, when using a hybrid mixed formulation of the flow subproblem, 
is expressed as: 
%turns into:
%following the definition
%of~\eqref{Cons-disc-3-field-HDG}, the resulting formulation now reads as: 
Find
$((\bu_h, \hat{\bu}_h), \bw_h, (p_h,\hat p_h) ) \in \overline{\bU}_h
\times \bm W^-_h \times \overline P_h$, such that
%for any
%$((\bv_h, \hat{\bv}_h), \bz_h, (q_h \hat q_h)) \in \overline{\bU}_h
%\times \bm W^-_h \times \overline P_h$ it holds that
%
\begin{subequations}\label{Cons-disc-3-field-HDG-HM}
\begin{alignat}{2}
 a_{h}((\bu_h, \hat{\bu}_h),(\bv_h, \hat{\bv}_h))- 
 (p_{h}, \divv \bv_h)
 &=(\bf, \bv_h), \quad &&\forall (\bv_h, \hat{\bv}_h) \in \overline{\bU}_h,\label{Cons-disc-3-field-HDG-HM_a} \\
  (R^{-1}\bw_{h},\bz_{h}) {-} b((p_{h}, \hat p_h), \bz_{h}) &= 0, \quad && \forall \bz_h \in  \bm W^-_h, \label{Cons-disc-3-field-HDG-HM_b} \\
 	 -\ (\divv \bu_h,q_{h}) - b((q_{h},\hat q_h),\bw_{h}) 
 	 -(S p_{h},q_{h}) 
 	 &= (g,q_{h}),  \quad &&\forall (q_h \hat q_h) \in  \overline P_h. \label{Cons-disc-3-field-HDG-HM_c}
\end{alignat}
\end{subequations}
Note that if we test this system with the test function $((\bm 0,\bm 0), \bm 0, (0,\hat q_h))$, we obtain
\begin{align*}
  b((0, \hat p_h),\bm w_h) = \sum_{T \in \mathcal{T}_h} \int_{\partial T} \bm w_h \cdot \bm n \hat p_h \ds = \sum_{F \in \mathcal{F}_h} \int_F [ \bm w_h \cdot n] \hat p_h \ds = 0.
\end{align*}
Hence, by choosing $\hat q_h = [ \bm w_h \cdot n]$ on each facet
$F \in \mathcal{F}_h$, the above equation demonstrates that the velocity
solution of \eqref{Cons-disc-3-field-HDG-HM} is normal continuous,
i.e. $\bm w_h \in \bm W_h$.

In the next section, we extend the parameter-robust stability
results from~\cite{HongKraus2018,Hong2019conservativeMPET} to the
hybridized three-field formulation given by systems~\eqref{Cons-disc-3-field-HDG} and~\eqref{Cons-disc-3-field-HDG-HM}.
%involving the bilinear
%form~\eqref{a_h_HDG}.
%This requires to define the
%proper parameter-depend norms for the facet space, resulting in
%parameter-robust norm-equivalent preconditioners and optimal optimal
%error estimates for the new variables.

\section{Parameter-robust stability, preconditioners and optimal error estimates}\label{sec:preco}

\subsection{Parameter-robust well-posedness}\label{sec::wellposed}

\subsubsection{Parameter-dependent norms}

%First, we present the parameter-dependent norms in which inf-sup conditions
%and continuity can be proven for the continuous problem as well as the
%different discretizations with constants independent of any physical
%parameters.

First, let us recall the norms previously used in the parameter robust stability analysis presented in~\cite{HongKraus2018}. 
These are, for the infinite dimensional spaces $\bm U, \bm W, P$, 
\begin{subequations}\label{norms_continuous_spaces}
\begin{align}
\Vert \bm v\Vert_{\bm U}^2 & := \Vert \eps (\bm v) \Vert^2_0 + \lambda \Vert \divv \bm v \Vert^2_0, \label{norm_U} \\
  \Vert \bm z\Vert^2_{\bm W} & := R^{-1} \Vert \bm z \Vert^2_0 + \gamma^{-1} \Vert \divv \bm z \Vert^2_0  , \label{norm_W} \\
  \Vert \bm z\Vert^2_{\bm W^{-}} & := R^{-1} \Vert \bm z \Vert^2_0  , \label{norm_W_minus} \\
\Vert q\Vert^2_{P} &:=  \gamma \Vert q \Vert^2_0, \label{norm_P}
\end{align}
\end{subequations}
where the parameter $\gamma$ can be defined as
$\gamma := \lambda_0^{-1}+R+S \eqsim \max\{\lambda_0^{-1},R,S\}$, with
$\lambda_0=\max\{1,\lambda\}\eqsim 1+\lambda$, or exactly as
in~\cite{HongKraus2018} where $\gamma$ has been been defined as
$ \gamma:=\max\{(\min\{\lambda,R^{-1}\})^{-1},S\}$. Due to the
non-conformity of the DG discretization, the norm for the discrete
displacement space $\bm U_h$ is based on the standard DG norm
\begin{equation}\label{norm_DG}
\Vert \bm v_h \Vert^2_{\text{DG}} :=  \sum_{T\in \mathcal{T}_h}\Vert \nabla \bm v_h \Vert_{0,T}^2 + \sum_{F\in \mathcal{F}_h}
h_F^{-1}\Vert [(\bm v_h)_t] \Vert_{0,F}^2
+  \sum_{T\in \mathcal{T}_h} h_T^2\vert  \bm v_h \vert_{2,T}^2
\end{equation}
and defined by
\begin{equation}\label{norm_Uh}
\Vert \bm v_h\Vert_{\bm U_h}^2: = \Vert \bm v_h \Vert^2_{\text{DG}} + \lambda \Vert \divv \bm v_h \Vert^2_0.
\end{equation}
Next, we introduce the hybridized discontinuous Galerkin (HDG) norm 
\begin{equation}\label{norm_HDG}
\Vert(\bm v_h, \hat{\bm v}_h) \Vert^2_{\text{HDG}} :=   \sum_{T\in \mathcal{T}_h} \left(\Vert \nabla \bm v_h \Vert_{0,T}^2 
+  h_T^{-1}\Vert (\hat{\bm v}_h - \bm v_h)_t \Vert_{0,\partial T}^2
+  h_T^2\vert  \bm v_h \vert_{2,T}^2
\right),
\end{equation}
based on which we can define a discrete norm on the extended displacement space $\overline{\bm U}_h$, i.e.,
\begin{equation}\label{norm_Uh_tilde}
\Vert (\bm v_h,\hat{\bm v}_h)\Vert_{\overline{\bm U}_h}^2 := \Vert (\bm v_h, \hat{\bm v}_h) \Vert^2_{\text{HDG}} 
+ \lambda \Vert \divv \bm v_h \Vert^2_0.
\end{equation}
Moreover, we define the following discrete norm on the extended pressure space $\overline{P}_h$
\begin{subequations} \label{hdg_p_norms}
\begin{align}
  \Vert (q_h, \hat{q}_h)\Vert^2_{\text{HDG}} &:= \sum_{T\in \mathcal{T}_h} \left(\Vert \nabla q_h \Vert_{0,T}^2 
+  h_T^{-1}\Vert \hat{ q}_h - q_h \Vert_{0,\partial T}^2
+  h_T^2\vert  q_h \vert_{2,T}^2\right)
%+ \Vert q_h \Vert^2_0
,\label{hdg_p_norm}\\
\Vert (q_h, \hat{q}_h)\Vert^2_{\overline{P}_h} &:=  R
% \sum_{T\in \mathcal{T}_h} \left(\Vert \nabla q_h \Vert_{0,T}^2 
%+  h_T^{-1}\Vert \hat{ q}_h - q_h \Vert_{0,\partial T}^2
%+  h_T^2\vert  q_h \vert_{2,T}^2\right)+ \gamma \Vert q_h \Vert^2_0.
 \Vert (q_h, \hat{q}_h)\Vert^2_{\text{HDG}} + \gamma  \Vert q_h \Vert^2_0,
\label{hdg_p_norm_scaled}
\end{align}
\end{subequations}
where 
\begin{equation}\label{gamma}
\gamma = S+\frac{1}{\lambda}\simeq \max\left\{S,\frac{1}{\lambda_0}\right\}.
\end{equation}
Finally, we consider the following two product spaces
\begin{subequations}\label{product_spaces}
\begin{eqnarray}
\overline{\bm X}_h &:=& \overline{\bm U}_h \times {\bm W_h}\times P_h \label{product_spaces_a}, \\
\overline{\overline{\bm X}}_h & := & \overline{\bm U}_h \times {\bm W_h^{-}}\times \overline{P}_h  \label{product_spaces_b}
\end{eqnarray}
\end{subequations}
equipped with the norms
\begin{subequations}\label{product_norms}
\begin{align}
\vertiii{((\bm v_h,\hat{\bm v}_h),\bm z_h,q_h)}^2_{\overline{\bm X}_h} &:=
 \Vert (\bm v_h,\hat{\bm v}_h) \Vert^2 _{\overline{\bm U}_h} + \Vert z_h\Vert^2_{\bm W} 
+\Vert q_h\Vert^2_P \label{product_norms_a}, \\
\vertiii{((\bm v_h,\hat{\bm v}_h),\bm z_h,(q_h,\hat{q}_h))}^2_{\overline{\overline{\bm X}}_h} &:=
 \Vert (\bm v_h,\hat{\bm v}_h) \Vert^2 _{\overline{\bm U}_h} + \Vert z_h\Vert^2_{\bm W^{-}} 
+\Vert (q_h,\hat{q}_h)\Vert^2_{\overline{P}_h}  \label{product_norms_b}
\end{align}
\end{subequations}
in the context of problems~\eqref{Cons-disc-3-field-HDG} and~\eqref{Cons-disc-3-field-HDG-HM}, respectively.

\subsubsection{Uniform well-posedness of the time-discrete problem}

The well-posedness of the three-field formulation \eqref{eqn:weak-3f}
on the continuous and discrete levels has been addressed and answered
in
\cite{Zenisek1984existence,Zenisek1984finite,Showalter2000diffusion,Hairer1989numerical}
using semi-group theory and Galerkin discretization methods.  After
time discretization by an implicit or semi-implicit time integration
scheme, the continuous three-field formulation results in a
variational problem of the form: Find $\bx\in \bm X$ such that
\begin{equation}\label{var_2_3}
\mathcal{A}(\bx,\by)=\mathcal{F}(\by), \qquad \forall \by \in \bm X := \bm U\times \bm W \times P,
\end{equation}
%where $\bm X = \bm U\times \bm W\times $ is a given Hilbert space, i.e., 
%$$\bm X = \begin{cases}\bm U \times P \qquad \quad\,\,\, \text{ for the two-field formulations} \\ 
%\bm U\times \bm W\times P \quad \text{ for the 3-field formulations},
%\end{cases}$$
where
$
  \mathcal{A}(\bx,\by)
  := 
  a(\bu,\bv)- 
  (p, \divv \bv) +
  (R^{-1}\bw,\bz) - (p,\divv \bz) 
  -\ (\divv \bu,q) - (\divv\bw,q) 
  -(S p,q) 
$
and $\mathcal{F}(\cdot) \in \bm X'$ denotes a corresponding linear
form, which depends on the time integrator.

As it is well known, the abstract variational problem~\eqref{var_2_3} is well-posed
under the following necessary and sufficient conditions, see~\cite{Babuska1971error}. 

\begin{theorem}\label{thm:1}
Assume that 
$\mathcal{F}\in \bm X'$ and the bilinear form $\mathcal{A}(\cdot,\cdot)$ in~\eqref{var_2_3} satisfies the following conditions:
\begin{itemize}
\item $\mathcal{A}(\cdot,\cdot)$ is bounded, i.e., there exists a constant $C>0$ such that 
\begin{equation}\label{boundedness_A}
\mathcal{A}(\bx,\by) \le C \vertiii{\bx}_{\bm X} \vertiii{\by}_{\bm X} \qquad \forall \bx,\by \in \bm X;
\end{equation}
\item There exists a constant $\beta>0$ such that 
\begin{equation}\label{inf_sup_A}
\inf_{\bx \in \bm X} \sup_{\by \in \bm X} \frac{\mathcal{A}(\bx,\by)}{\vertiii{\bx}_{\bm X} \vertiii{\by}_{\bm X}} \ge \beta>0.
\end{equation}
\end{itemize}
Then there exists a unique solution $\bx^* \in \bm X$ of the variational problem~\eqref{var_2_3}. Further, the solution 
satisfies the stability estimate 
$$
\vertiii{\bx^*}_{\bm X} \le \sup_{\by \in \bm X}\frac{\mathcal{F}(\by)}{\vertiii{\by}_{\bm X}}
=:\Vert \mathcal{F}\Vert_{\bm X'}.
$$
\end{theorem}

Besides for the establishment of well-posedness on the continuous and discrete levels, 
boundedness, i.e., property~\eqref{boundedness_A}, and inf-sup stability, i.e., property~\eqref{inf_sup_A}, 
is crucial in the error analysis and for the construction of preconditioners and iterative solution
methods for the algebraic problems arising from the discretization
of~\eqref{var_2_3}. 
Furthermore, aiming at parameter-independent error, or
near-best approximation estimates and parameter-robust
preconditioners, it is essential that the constants $C$ and $\beta$
in~\eqref{boundedness_A} and \eqref{inf_sup_A} are independent of any
physical (model) and discretization parameters. 
%This, however, is only
%possible when using a proper parameter-dependent norm, which, upon
%allowing the estimates \eqref{boundedness_A} and \eqref{inf_sup_A} to
%be parameter-independent, can be considered as a 'natural' norm for
%measuring the discretization error.
\begin{definition}
We call problem~\eqref{var_2_3} uniformly well-posed on its parameter space (or, in short, uniformly well-posed) 
under the norm $\vertiii{\cdot}_{\bm X}$
if the conditions of Theorem~\ref{thm:1} are satisfied and the constants $C$ and $\beta$ in~\eqref{boundedness_A} 
and \eqref{inf_sup_A} do not depend on any of the problem parameters. 
\end{definition}
\begin{remark}
The parameter space is the space of all problem parameters, i.e., physical parameters of the 
continuous mathematical model 
but also discretization parameters when $\mathcal{A}(\cdot,\cdot)$ represents a semi- or fully discrete problem.
\end{remark}
Uniform well-posedness of the time-discrete problem resulting from the three-field formulation of Biot's consolidation model 
has first been proven in~\cite{HongKraus2018} 
%
%For the time-discrete problem resulting from the three-field formulation, this problem was first solved in~\cite{HongKraus2018} 
using the norm 
\begin{equation}\label{norms_HK}
\vertiii{(\bv,\bz,q)}_{\bm X}^2 := \Vert \bv\Vert_{\bm U}^2 + \Vert \bz\Vert_{\bm W}^2 + \Vert q\Vert_P^2
\end{equation}
where $\Vert \cdot \Vert_{\bm U}$, $\Vert \cdot \Vert_{\bm W}$,
$\Vert \cdot\Vert_P$ are defined in~\eqref{norms_continuous_spaces}.
In the remainder of Subsection~\ref{sec::wellposed}, we extend the uniform
well-posedness analysis from~\cite{HongKraus2018,Hong2019conservativeMPET} to the three-field
formulations~\eqref{Cons-disc-3-field-HDG} and~\eqref{Cons-disc-3-field-HDG-HM}.

\subsubsection{Hybridized DG method}

Following the approach presented in~~\cite{HongKraus2018}, we will
show that problem~\eqref{Cons-disc-3-field-HDG} is uniformly
well-posed.  Initially, we rewrite \eqref{Cons-disc-3-field-HDG} in
the form: Find
$\bar{\bx}_h:=((\bu_h,\hat{\bu}_h),\bw_h,p_h)\in \overline{\bm U}_h
\times \bm W_h \times P_h =:\overline{\bm X}_h$, such that
%for any
%$((\bv_h,\hat{\bv}_h),\bz_h,q_h)=:\bar{\by}_h\in \overline{\bm X}_h$
%it holds that
\begin{equation}\label{HDG_gen}
\overline{\mathcal{A}}_h(\bar{\bx}_h,\bar{\by}_h)=\overline{\mathcal{F}}_h(\bar{\by}_h),\qquad \forall \bar{\by}_h\in \overline{\bm X}_h,
\end{equation}
where with $ \bar{\by}_h :=((\bv_h,\hat{\bv}_h),\bz_h,q_h)$ we have
\begin{subequations}\label{HDG_both}
\begin{align}
\overline{\mathcal{A}}_h(\bar{\bx}_h,\bar{\by}_h) : = & a_h((\bu_h,\hat{\bu}_h),(\bv_h,\hat{\bv}_h))-(p_h,\divv \bv_h)+(R^{-1}\bw_h,\bz_h) \nonumber \\
& -(p_h,\divv \bz_h) - (\divv \bu_h,q_h)-(\divv \bw_h,q_h) - (S p_h,q_h),  \label{HDG_left} \\
\overline{\mathcal{F}}_h(\bar{\by}_h) := &(\bm f,\bv_h)+(g,q_h), \label{HDG_right}
\end{align}
\end{subequations}
and $a_h((\cdot,\cdot),(\cdot,\cdot))$ is defined in~\eqref{a_h}.
Next, we recall two auxiliary results crucial for establishing the
main result of this subsection.
\begin{lemma}\label{lemma_inf_sup}
The following discrete inf-sup condition
\begin{align}
\inf_{q_h \in P_h}\sup_{(\bv_h,\hat{\bv}_h)\in \overline{\bm U}_h}
\frac{(\emph{div} \bv_h,q_h)}{\Vert (\bv_h,\hat{\bv}_h)\Vert_{\emph{HDG}}\Vert q_h\Vert_0} 
\ge \bar{\beta}_{S,d}>0,
  \label{hdg_inf_sup_stokes}
%\inf_{q_h \in P_h}\sup_{\bz_h \in \bm V_h}
%\frac{(\emph{div} \bz_h,q_h)}{\Vert \bz_h\Vert_{\emph{div}}\Vert q_h\Vert_0} 
%\ge \bar{\beta}_{D,d}>0
%\label{hdg_inf_sup_darcy} 
\end{align}
%where $\Vert \cdot \Vert_{\emph{div}}$ denotes the $H(\emph{div})$-norm defined by 
%\begin{equation}\label{hdiv_norm}
%\Vert \bz_h\Vert_{\emph{div}}^2 = \Vert \emph{div} \bz_h \Vert^2_0 +\Vert \bz_h\Vert_0^2
%\end{equation}
%and
holds where $\Vert \cdot \Vert_{\emph{HDG}}$ is the HDG norm defined in~\eqref{norm_HDG}.
\end{lemma}

\begin{proof}
As shown, for example, in~\cite{Hong2016arobust,Hong2016uniformly}, the following inf-sup condition holds true:
\begin{equation}\label{inf_sup_help}
\inf_{q_h \in P_h}\sup_{\bv_h\in \bm U_h}
\frac{(\divv \bv_h,q_h)}{\Vert \bv_h\Vert_{\text{DG}}\Vert q_h\Vert_0} 
\ge \beta_{S,d}>0.
\end{equation}
Moreover, for all $\bv_h\in \bm U_h$ there exists $\hat{\bv}_h \in \widehat{\bU}_h$ such that
$
\Vert \bv_h\Vert_{\text{DG}} \ge C \Vert (\bv_h,\hat{\bv}_h)\Vert_{\text{HDG}}
$
with a constant $C$ depending only on mesh regularity. Combining the latter estimate with \eqref{inf_sup_help} 
yields~\eqref{hdg_inf_sup_stokes}. 
%The Darcy inf-sup condition \eqref{hdg_inf_sup_darcy} is well known and a proof can be found, for example, in 
%\cite{Boffi2013mixed}.
\end{proof}

The proof of the following theorem also makes use of the boundedness and coercivity of the bilinear form 
$a_h^{\text{HDG}}((\cdot,\cdot),(\cdot,\cdot))$ on $\overline{\bU}_h$ defined in~\eqref{a_h_HDG}, i.e., 
\begin{equation}\label{a_h_HDG_cont}
\vert a_h^{\text{HDG}}((\bu_h,\hat{\bu}_h),(\bv_h,\hat{\bv}_h)) \vert \le C_a \Vert (\bu_h,\hat{\bu}_h) \Vert_{\text{HDG}}
\Vert (\bv_h,\hat{\bv}_h) \Vert_{\text{HDG}}
% \; \forall (\bu_h,\hat{\bu}_h), (\bv_h,\hat{\bv}_h)\in \overline{\bm U}_h.
\end{equation} 
for all $(\bu_h,\hat{\bu}_h), (\bv_h,\hat{\bv}_h)\in \overline{\bm U}_h$ and
\begin{equation}\label{a_h_HDG_coer}
 a_h^{\text{HDG}}((\bu_h,\hat{\bu}_h),(\bu_h,\hat{\bu}_h))  \ge C_c \Vert (\bu_h,\hat{\bu}_h) \Vert_{\text{HDG}}^2
 \quad \text{for all } (\bu_h,\hat{\bu}_h)\in \overline{\bm U}_h,
\end{equation} 
see e.g.~\cite{LehrenfeldSchoeberl2016high,LedererEtAl2018hybrid}.
%
%Then we have the following theorem: 
%Theorem~\ref{well-posed_hdg} holds true
\begin{theorem}\label{well-posed_hdg}
Problem~\eqref{HDG_gen}--\eqref{HDG_both} is uniformly well-posed under the norm $\vertiii{\cdot}_{\overline{\bm X}_h}$ defined 
in~\eqref{product_norms_a}, that is,
\begin{equation}\label{A_hdg_bounded}
\overline{\mathcal{A}}(\bar{\bx}_h,\bar{\by}_h) 
\le \overline{C} \vertiii{\bar{\bx}_h}_{\overline{\bm X}_h} \vertiii{\bar{\by}_h}_{\overline{\bm X}_h} 
\qquad \forall \bar{\bx}_h,\bar{\by}_h \in \overline{\bm X}_h,
\end{equation}
\begin{equation}\label{A_hdg_infsup}
\inf_{\bar{\bx}_h\in \overline{\bm X}_h} \sup_{\bar{\by}_h\in \overline{\bm X}_h}  
\frac{\overline{\mathcal{A}}(\bar{\bx}_h,\bar{\by}_h) }{\vertiii{\bar{\bx}_h}_{\overline{\bm X}_h} \vertiii{\bar{\by}_h}_{\overline{\bm X}_h} } 
\ge \bar{\beta} >0.
\end{equation}
%where the constants $\overline{C}$ and $ \bar{\beta}$ are independent
%of the parameters $\lambda$, $R$, $S$ and $\tau,h$.
\end{theorem}

\begin{proof}

To show~\eqref{A_hdg_bounded} one uses Cauchy-Schwarz inequality, 
the continuity of the bilinear form $a_h((\cdot,\cdot),(\cdot,\cdot))$ in the norm $\Vert \cdot\Vert_{\overline{\bU}_h}$
on $\overline{\bm U}_h$, i.e.,
\begin{equation}\label{a_h_cont}
\vert a_h((\bu_h,\hat{\bu}_h),(\bv_h,\hat{\bv}_h)) \vert \le \overline{C}_a \Vert (\bu_h,\hat{\bu}_h) \Vert_{\overline{\bm U}_h}
\Vert (\bv_h,\hat{\bv}_h) \Vert_{\overline{\bm U}_h} \quad \forall (\bu_h,\hat{\bu}_h), (\bv_h,\hat{\bv}_h)\in \overline{\bm U}_h,
\end{equation} 
which follows from~\eqref{a_h_HDG_cont}
%, see e.g.~\cite{LehrenfeldSchoeberl2016high,LedererEtAl2018hybrid}, 
as well as the definitions of $a_h((\cdot,\cdot),(\cdot,\cdot))$, 
$\Vert \cdot \Vert_{\overline{\bm U}_h}$ and $\vertiii{\cdot}_{\overline{\bm X}_h}$, 
see~\eqref{a_h}, \eqref{norm_Uh_tilde} and \eqref{product_norms_a}, respectively.

The proof of~\eqref{A_hdg_infsup} follows exactly the lines of the proof of Theorem~4.4 in \cite{HongKraus2018} 
replacing the DG bilinear form~\eqref{a_h_DG} by the HDG bilinear form \eqref{a_h_HDG} and the DG norm 
\eqref{norm_DG} by the HDG norm \eqref{norm_HDG}.
\end{proof}
% Note that uniform well-posedness proven theorem \ref{well-posed_hdg}
% also means that the constants are independent not only of the
% physical parameters $\lambda$, $R$, $S$ but also of the
% discretization parameters $\tau$ and $h$.

\subsubsection{Hybridized DG/hybrid mixed method}

%Let us now turn to 
Consider the HDG/hybrid mixed method for the three-field formulation as stated 
in~\eqref{Cons-disc-3-field-HDG-HM}. To prove the uniform well-posedness of this fully discrete problem, 
as we did with \eqref{Cons-disc-3-field-HDG}, we rewrite \eqref{Cons-disc-3-field-HDG-HM} in the form: 
Find $\overline{\overline{\bx}}_h:=((\bu_h,\hat{\bu}_h),\bw_h,(p_h,\hat{p}_h))\in \overline{\bU}_h\times \bW_h\times \overline{P}_h 
=:\overline{\overline{\bm X}}_h$ such that 
%for any $((\bv_h,\hat{\bv}_h),\bz_h,(q_h,\hat{q}_h))=:\overline{\overline{\by}}_h\in \overline{\overline{\bm X}}_h$ it holds that 
\begin{equation}\label{HDG-HM_gen}
\overline{\overline{\mathcal{A}}}_h
(\overline{\overline{\bx}}_h,\overline{\overline{\by}}_h)
=\overline{\overline{\mathcal{F}}}_h(\overline{\overline{\by}}_h),\qquad \forall \overline{\overline{\by}}_h\in 
\overline{\overline{\bm X}}_h,
\end{equation}
where with
$\overline{\overline{\by}}_h :=
((\bv_h,\hat{\bv}_h),\bz_h,(q_h,\hat{q}_h))$ we have
\begin{subequations}\label{HDG-HM_both}
\begin{align}
\overline{\overline{\mathcal{A}}}_h
(\overline{\overline{\bx}}_h,\overline{\overline{\by}}_h) : = & a_h((\bu_h,\hat{\bu}_h),(\bv_h,\hat{\bv}_h))-(p_h,\divv \bv_h)+(R^{-1}\bw_h,\bz_h)  \nonumber \\
%&  \nonumber \\
& - b((p_h,\hat{p}_h),\bz_h)- (\divv \bu_h,q_h)- b((q_h,\hat{q}_h),\bw_h) - (S p_h,q_h),  \label{HDG-HM_left} \\
\overline{\overline{\mathcal{F}}}_h(\overline{\overline{\by}}_h) := &(\bm f,\bv_h)+(g,q_h), \label{HDG-HM_right}
\end{align}
\end{subequations}
and $a_h((\cdot,\cdot),(\cdot,\cdot))$ and $b((\cdot,\cdot),\cdot)$
are defined in~\eqref{a_h} and \eqref{form_b}, respectively.  Before proving the
main theorem, we need another auxiliary result given by the following
lemma.
\begin{lemma}\label{lemma_HDG-HM}
There holds the following discrete inf-sup condition
%\begin{subequations}
%\begin{align}
\begin{equation}\label{hdg-hm_inf_sup_darcy}
%\inf_{q_h \in P_h}\sup_{(\bv_h,\hat{\bv}_h) \in \overline{\bm U}_h}
%\frac{(\emph{div} \bv_h,q_h)}{\Vert (\bv_h,\hat{\bv}_h)\Vert_{\emph{HDG}}\Vert q_h\Vert_0} 
%\ge \overline{\overline{\beta}}_{S,d}>0,\label{hdg-hm_inf_sup_stokes}\\
\inf_{(q_h,\hat{q}_h) \in \overline{P}_h}\sup_{\bz_h \in\bm V_h}
\frac{b((q_h,\hat{q}_h),\bz_h)}{\Vert \bz_h\Vert_{0}\Vert (q_h, \hat q_h)\Vert_{\emph{HDG}}}
  \ge {\overline{\beta}}_{D,d}>0
 \end{equation} 
 %\end{align}
%\end{subequations}
where $\Vert(\cdot,\cdot) \Vert_{\emph{HDG}} $ is defined in~\eqref{hdg_p_norm}.
\end{lemma}
\begin{proof}
A direct proof of~\eqref{hdg-hm_inf_sup_darcy} can be readily constructed, 
%following
%, for example, 
%the steps 
similarly as for the inf-sup condition in~\cite{Lederer:2019b,Lederer:2019c}, 
using the definition of the degrees of freedom for the Raviart-Thomas space, see~\cite{Boffi2013mixed}, and standard scaling arguments.
%This inf-sup condition is discussed and used, %. We want to refer for example to the works 
%e.g. in~\cite{MR2051067, CockburnEtAl2009unified}. 
\end{proof}
Such inf-sup conditions with mesh-dependent norms are 
widely used in structural mechanics, see, e.g.,~\cite{veubeke}. 
%\jknote{analysis}
%\mlnote{analysis}
%\plnote{check for standard LBB proof/Lemma of $b((p_h \hat p_h), w_h)$ with DG like norm for $(p_h \hat p_h)$ and $L^2$ norm for $w_h$. In total we have three LBBs: Stokes, Darcy and $b((p_h \hat p_h), w_h)$. }
\begin{theorem}\label{well-posed_hdg-hm}
Problem~\eqref{HDG-HM_gen}--\eqref{HDG-HM_both} is 
uniformly well-posed under the norm $\vertiii{\cdot}_{\overline{\overline{\bm X}}_h}$ defined 
in~\eqref{product_norms_b}.
%with constants independent
%of the parameters $\lambda$, $R$, $S$ and $\tau,h$.
\end{theorem}

\begin{proof}
We start with proving the boundedness of the bilinear form $\overline{\overline{\mathcal{A}}}_h(\cdot,\cdot)$, i.e.,
\begin{equation}\label{bound_A_barbar}
\overline{\overline{\mathcal{A}}}_h( \overline{\overline{\bx}}_h, \overline{\overline{\by}}_h) \le 
\overline{\overline{C}} \vertiii{\overline{\overline{\bx}}_h}_{\overline{\overline{\bm X}}_h} 
\vertiii{\overline{\overline{\by}}_h}_{\overline{\overline{\bm X}}_h} \qquad \forall \overline{\overline{\bx}}_h, 
\overline{\overline{\by}}_h\in \overline{\overline{\bm X}}_h.
\end{equation}

First we note that 
\begin{align}
  b((p_h,\hat{p}_h),\bw_h) &= \sum_{T \in \mathcal{T}_h} \int_T \div \bm w_h p_h \dx - \int_{\partial T} \bm w_h \cdot \bm n \hat p_h \ds \nonumber\\
                           &= \sum_{T \in \mathcal{T}_h} \int_T -\bm w_h \cdot \nabla p_h \dx - \int_{\partial T} \bm w_h \cdot \bm n (\hat p_h - p_h) \ds \nonumber\\
                           &= \sum_{T \in \mathcal{T}_h} \int_T -\bm w_h \cdot \nabla p_h \dx - \int_{\partial T} h \bm  w_h \cdot \bm n \frac{1}{h} (\hat p_h - p_h) \ds \nonumber\\
                           &\le \sqrt{ \Vert\bm w_h \Vert^2_0 + \sum_{T \in \mathcal{T}_h} h \Vert\bm w_h \cdot n\Vert_{\partial T} } \Vert (p_h, \hat p_h) \Vert_{\text{HDG}} \nonumber\\
  &\le C_b\Vert \bm w_h \Vert_0 \Vert (p_h, \hat p_h) \Vert_{\text{HDG}}, \label{eq::bcont}
\end{align}
where we have used standard scaling arguments in the last step of~\eqref{eq::bcont}, i.e. the constant $C_b$ depends only on the mesh 
regularity.

Further, using the continuity of the bilinear form $a_h((\cdot,\cdot),(\cdot,\cdot))$ on $\overline{\bm U}_h$, 
i.e.~\eqref{a_h_cont}, the definitions of the norms $\Vert \cdot\Vert_{\overline{\bU}_h}$, $\Vert \cdot \Vert_W$, 
$\Vert \cdot\Vert_{\overline{P}_h}$, and $\vertiii{\cdot}_{\overline{\overline{\bm X}}_h}$, 
see \eqref{norm_Uh_tilde}, \eqref{norm_W}, \eqref{hdg_p_norm_scaled} and \eqref{product_norms_b}, 
respectively, and applying the Cauchy-Schwarz inequality and also estimate~\eqref{eq::bcont}, one gets
\begin{align}
\overline{\overline{\mathcal{A}}}_h
(\overline{\overline{\bx}}_h,\overline{\overline{\by}}_h) : = 
     & a_h((\bu_h,\hat{\bu}_h),(\bv_h,\hat{\bv}_h))-(p_h,\divv \bv_h)+(R^{-1}\bw_h,\bz_h) \nonumber \\
%&  \nonumber \\
&  - b((p_h,\hat{p}_h),\bz_h)- (\divv \bu_h,q_h)- b((q_h,\hat{q}_h),\bw_h) - (S p_h,q_h)  \nonumber \\
\le & C_a \Vert (\bu_h,\hat{\bu}_h)\Vert_{\overline{\bU}_h} \Vert (\bv_h,\hat{\bv}_h)\Vert_{\overline{\bU}_h} 
+\lambda^{-1/2}\Vert p_h\Vert_0 \lambda^{1/2} \Vert \divv \bv_h \Vert_0  \nonumber \\
& + R^{-1/2} \Vert \bw_h \Vert_0 R^{-1/2} \Vert \bz_h\Vert_0 
+ C_bR^{1/2}   \Vert (p_h,\hat{p}_h) \Vert_{\text{HDG}} 
R^{-1/2}\Vert  \bz_h \Vert_0
 \nonumber \\
& + \lambda^{1/2}\Vert \divv \bu_h\Vert_0 \lambda^{-1/2} \Vert q_h \Vert_0 
+ C_bR^{1/2}  \Vert (q_h,\hat{q}_h) \Vert_{\text{HDG}}
R^{-1/2}\Vert  \bw_h \Vert_0  \nonumber \\
& + S^{1/2} \Vert p_h\Vert_0 S^{1/2} \Vert q_h \Vert_0  \nonumber \\
\le & C_a \Vert (\bu_h, \hat{\bu}_h)\Vert_{\overline{\bU}_h}\Vert (\bv_h,\hat{\bv}_h\Vert_{\overline{\bU}_h} 
+ \Vert (p_h,\hat{p}_h) \Vert_{\overline{P}_h} \Vert (\bv_h, \hat{\bv}_h)\Vert_{\overline{\bU}_h}  \nonumber \\
& + \Vert \bw_h \Vert_{\bW^-} \Vert \bz_h \Vert_{\bW^-} + C_b\Vert (p_h, \hat{p}_h)\Vert_{\overline{P}_h}\Vert \bz_h\Vert_{\bW^-} 
 \nonumber \\
& + \Vert (\bu_h, \hat{\bu}_h)\Vert_{\overline{\bU}_h}  \Vert (q_h,\hat{q}_h) \Vert_{\overline{P}_h} 
+  C_b\Vert (q_h,\hat{q}_h)\Vert_{\overline{P}_h}\Vert \bw_h\Vert_{\bW^-}  \nonumber \\
& +  \Vert (p_h,\hat{p}_h) \Vert_{\overline{P}_h}  \Vert (q_h,\hat{q}_h) \Vert_{\overline{P}_h}  \nonumber \\
\le & \overline{\overline{C}} \left( \Vert (\bu_h,\hat{\bu}_h) \Vert_{\overline{\bU}_h} 
+\Vert \bw_h\Vert_{\bW^-} + \Vert (p_h,\hat{p}_h)\Vert_{\overline{P}_h} \right)  \nonumber \\
& \times
\left( \Vert (\bv_h,\hat{\bv}_h) \Vert_{\overline{\bU}_h} 
+\Vert \bz_h\Vert_{\bW^-} + \Vert (q_h,\hat{q}_h)\Vert_{\overline{P}_h} \right).
\end{align}

Next we prove the inf-sup condition 
\begin{equation}\label{A_hdg_mixed_infsup}
\inf_{\overline{\overline{\bx}}_h\in \overline{\overline{\bm X}}_h} \sup_{\overline{\overline{\by}}_h\in 
\overline{\overline{\bm X}}_h}  
\frac{\overline{\overline{\mathcal{A}}}_h(\overline{\overline{\bx}}_h,
\overline{\overline{\by}}_h) }
{\vertiii{\overline{\overline{\bx}}_h}_{\overline{\overline{\bm X}}_h} 
\vertiii{\overline{\overline{\by}}_h}_{\overline{\overline{\bm X}}_h} } 
\ge \overline{\overline{\beta}} >0
\end{equation}
which immediately follows if for all $\oox_h\in \ooX_h$ we can find 
$\ooy_h = \ooy_h(\oox_h)$ such that
\begin{equation}\label{bound_est}
\vertiii{\ooy_h}_{\ooX_h} \le \overline{\overline{C}}_b \vertiii{\oox_h}_{\ooX_h} %\qquad \text{for all }\oox_h\in \ooX_h
\end{equation}
and the coercivity estimate 
\begin{equation}\label{coerc_est}
\overline{\overline{\mathcal{A}}_h}(\oox_h,\ooy_h) \ge \overline{\overline{C}}_c 
\vertiii{\oox_h}^2_{\ooX_h} %\qquad \text{for all }\oox_h\in \ooX_h
\end{equation}
are simultaneously satisfied with constants
$\overline{\overline{C}}_b$ and $\overline{\overline{C}}_c$
independent of all problem parameters.

Now let $\oox_h\in \ooX_h$ be arbitrary but fixed. Then we choose
$\ooy_h:=((\bv_h,\hv_h),\bz_h,(q_h,\hq_h))$ by setting
%
%We start with verifying the bound~\eqref{bound_est} which we rewrite in the form 
%\begin{equation}\label{bound_est1}\nonumber
%\vertiii{((\bv_h,\hat{\bv}_h),\bz_h,(q_h,\hat{q}_h))}_{\ooX} \le \overline{\overline{C}}_b 
%\vertiii{((\bu_h,\hat{\bu}_h),\bw_h,(p_h,\hat{p}_h))}_{\ooX}  \quad \forall ((\bu_h,\hat{\bu}_h),\bw_h,(p_h,\hat{p}_h))\in \ooX_h.
%\end{equation}
%
%...
%
%Choose
\begin{subequations}\label{choices}
\begin{align}
(\bv_h,\hv_h)&:=\delta(\bu_h,\hu_h)-\frac{1}{\sqrt{\lambda_0}}(\bu_{h,0},\hu_{h,0}), \label{test_f_v}  \\
\bz_h & := \delta \bw_h + R \bw_{h,0}, \label{test_f_z}\\
(q_h,\hq_h) & := -\delta (p_h,\hp_h), \label{test_f_q}
\end{align}
\end{subequations}
where 
$(\bu_{h,0},\hat{\bu}_{h,0})\in \overline{\bU}_h$ is such that 
\begin{subequations}\label{u_h0}
\begin{align}
\divv \bu_{h,0}&=\frac{1}{\sqrt{\lambda_0}}p_h, \label{u_h0_a}\\
\Vert (\bu_{h,0},\hat{\bu}_{h,0})\Vert_{\text{HDG}}& \le \bar{\beta}_{S,d}^{-1} \frac{1}{\sqrt{\lambda_0}} \Vert p_h\Vert_0
\label{u_h0_b}
\end{align}
\end{subequations}
and $\bw_{h,0}$ is such that 
\begin{subequations}\label{w_h0}
\begin{align}
-  b((p_h,\hat{p}_h),\bw_{h,0})&=  \Vert (p_h,\hat{p}_h)\Vert^2_{\text{HDG}}, \label{w_h0_a} \\
\Vert \bw_{h,0}\Vert_0 & \le \overline{\beta}^{-1}_{D,d} \Vert (p_h,\hat{p}_h) \Vert_{\text{HDG}} \label{w_h0_b}.
\end{align}
\end{subequations}
Note that the existence of $(\bu_{h,0},\hat{\bu}_{h,0})$ and
$\bw_{h,0}$ satisfying the estimates~\eqref{u_h0} and \eqref{w_h0}
follows from the discrete inf-sup conditions~\eqref{hdg_inf_sup_stokes} and~\eqref{hdg-hm_inf_sup_darcy}. With
this particular choice, we first verify~\eqref{bound_est}. To begin with
\begin{subequations}
\allowdisplaybreaks
\begin{align}
\Vert \frac{1}{\sqrt{\lambda}_0} (\bu_{h,0},\hu_{h,0})\Vert^2_{\overline{\bU}_h}= &
\Vert \frac{1}{\sqrt{\lambda_0}}(\bu_{h,0},\hu_{h,0})\Vert^2_{\text{HDG}} 
+\lambda_0(\divv \left(\frac{1}{\sqrt{\lambda_0}}\bu_{h,0}\right),
\divv \left(\frac{1}{\sqrt{\lambda}_0}\bu_{h,0}\right)) \nonumber \\
\overset{\eqref{u_h0_b}}{\le} & \frac{1}{\lambda_0} 
\overline{\beta}_{S,d}^{-2}  \frac{1}{\lambda_0}  \Vert p_{h,0}\Vert_0^2 
+ \frac{1}{\lambda_0}  \Vert p_h\Vert^2_0 \nonumber 
\\
\le & \left( \frac{1}{\lambda_0} \overline{\beta}_{S,d}^{-2}+1 \right)
\gamma \Vert p_h\Vert^2_0 \nonumber \\
\le &  \left( \frac{1}{\lambda_0} \overline{\beta}_{S,d}^{-2}+1 \right) 
\Vert (p_h,\hp_h)\Vert^2_{\overline{P}_h}, \nonumber
\end{align}
\end{subequations}
from which we conclude
\begin{equation}\label{bound_a}
\Vert(\bv_h,\hv_h)\Vert_{\overline{\bU}_h} \le \delta 
\Vert(\bu_h,\hu_h)\Vert_{\overline{\bU}_h}+(\overline{\beta}_{S,d}^{-2}+1)^{\frac{1}{2}} 
\Vert (p_h,\hp_h)\Vert_{\overline{P}_h}.
\end{equation}
Next, 
%\begin{subequations}
\allowdisplaybreaks
\begin{align}
\Vert \bz_h \Vert^2_{{\bW}^{-}} \le & \delta \Vert \bw_h\Vert_{\bW^{-}}
+ R\Vert \bw_{h,0}\Vert_{\bW^{-}} \nonumber \\
\le & \delta \Vert \bw_h \Vert_{\bW^{-}}+\sqrt{R}\Vert \bw_{h,0}\Vert_0
\nonumber \\ 
\overset{\eqref{w_h0_b}}{\le} & \delta \Vert \bw_h \Vert_{\bW^{-}}
+\sqrt{R} \overline{\beta}_{D,d}^{-1} \Vert (p_h,\hp_h)\Vert_{HDG}
 \nonumber \\ \label{bound_b}
\le & 
\delta \Vert \bw_h \Vert_{\bW^{-}}
+ \overline{\beta}_{D,d}^{-1} \Vert (p_h,\hp_h)\Vert_{\overline{P}_h}. %\tag{4.41}
\end{align}
%\end{subequations}
Finally, 
\begin{equation}\label{bound_c}
\Vert (q_h,\hq_h)\Vert_{\overline{P}_h}\le \delta \Vert (p_h,\hp_h)\Vert_{\overline{P}_h}.
\end{equation}
The bounds \eqref{bound_a}, \eqref{bound_b} and \eqref{bound_c}
together imply \eqref{bound_est} with
$ \overline{\overline{C}}_b=[2(\delta^2+\overline{\beta}_{S,d}^{-2}
+\overline{\beta}_{D,d}^{-2}+1 )]^{\frac{1}{2}}$.

What remains is to verify~\eqref{coerc_est}:
%Then
\begin{subequations}
\allowdisplaybreaks
\begin{align}
\ooA_h((\bu_h,\hu_h),\bw_h,(p_h,\hp_h)) =& a_h^{\text{HDG}}((\bu_h,\hu_h),(\bv_h,\hv_h))+\lambda(\divv \bu_h,\divv \bv_h) 
\nonumber \\
& -(p_h,\divv \bv_h) +R^{-1}(\bw_h,\bz_h)-b((p_h,\hp_h),\bz_h)-(\divv \bu_h,q_h)\nonumber \\
&-b((q_h,\hq_h),\bw_h)-(Sp_h,q_h) \nonumber \\
%%
%= & a_h^{\text{HDG}}((\bu_h,\hu_h),
%\delta(\bu_h,\hu_h)-\frac{1}{\sqrt{\lambda_0}}(\bu_{h,0},\hu_{h,0})) \nonumber \\
%& +\lambda(\divv \bu_h,\delta\divv \bu_h-\frac{1}{\sqrt{\lambda_0}}\divv \bu_{h,0}) 
%\nonumber \\
%& -(p_h,\delta\divv \bu_h-\frac{1}{\sqrt{\lambda_0}}\divv \bu_{h,0}) 
%+R^{-1}(\bw_h,\delta \bw_h+R \bw_{h,0}) \nonumber \\
%& -\delta b((p_h,\hp_h), \bw_h) -R b((p_h,\hp_h),\bw_{h,0})
%+\delta (\divv \bu_h,p_h)\nonumber \\
%&+\delta b((p_h,\hp_h),\bw_h)+\delta (Sp_h,p_h) \nonumber \\
%%
= & \delta a_h^{\text{HDG}}((\bu_h,\hu_h),
(\bu_h,\hu_h))-\frac{1}{\sqrt{\lambda_0}} a_h^{HDG}
((\bu_h,\hu_h),(\bu_{h,0},\hu_{h,0})) \nonumber \\
& +\delta \lambda(\divv \bu_h,\divv \bu_h) 
-\frac{\lambda}{\sqrt{\lambda_0}}
(\divv \bu_h,\divv \bu_{h,0})  -\delta(p_h,\divv \bu_h)
\nonumber \\
%& 
%\nonumber
%\\
&+\frac{1}{\sqrt{\lambda_0}}(p_h,\divv \bu_{h,0})  +\delta R^{-1}(\bw_h, \bw_h) 
+(\bw_h,\bw_{h,0})
\nonumber \\
&  + R \Vert (p_h,\hp_h)\Vert_{\text{HDG}}^2
+\delta (\divv \bu_h,p_h) +\delta (Sp_h,p_h) \nonumber \\
\ge & \delta a_h^{\text{HDG}}((\bu_h,\hu_h),
(\bu_h,\hu_h))-\frac{1}{2}\frac{1}{\lambda_0} \varepsilon_1^{-1} a_h^{\text{HDG}}
((\bu_h,\hu_h),(\bu_{h},\hu_{h}))
 \nonumber \\
&
-\frac{1}{2} \varepsilon_1 a_h^{\text{HDG}}
((\bu_{h,0},\hu_{h,0}),(\bu_{h,0},\hu_{h,0}))+\delta \lambda(\divv \bu_h,\divv \bu_h) 
 \nonumber \\
& 
-\frac{1}{2}\varepsilon_2^{-1}\lambda
(\divv \bu_h,\divv \bu_{h}) 
-\frac{1}{2}\varepsilon_2 \frac{\lambda}{\lambda_0}
(\divv \bu_{h,0},\divv \bu_{h,0}) 
\nonumber \\
& 
+\frac{1}{\lambda_0}(p_h,p_h)  +\delta R^{-1}(\bw_h, \bw_h)-\frac{1}{2}\varepsilon_3^{-1}R^{-1}(\bw_h,\bw_{h})
\nonumber
\\
%& 
%\nonumber \\
& -\frac{1}{2}\varepsilon_3 R(\bw_{h,0},\bw_{h,0}) + R \Vert (p_h,\hp_h)\Vert_{\text{HDG}}^2
+\delta (Sp_h,p_h) \nonumber \\
\ge & \left(\delta-\frac{1}{2}\frac{1}{\lambda_0}\varepsilon_1^{-1}\right)
a_h^{\text{HDG}}((\bu_h,\hu_h),
(\bu_h,\hu_h))+\left(\delta - \frac{1}{2}\varepsilon_2^{-1}\right)\lambda(\divv \bu_h,\divv \bu_h) \nonumber \\
%& 
%\nonumber \\
& + \left(\delta S+\frac{1}{\lambda_0}-\frac{1}{2}\varepsilon_2 
\frac{\lambda}{\lambda_0^2}-\frac{1}{2}\frac{1}{\lambda_0}
\varepsilon_1 C_a \beta_{S,d}^{-2}\right) (p_h,p_h) \nonumber \\
%&  \nonumber 
%\\ 
&+ \left(\delta-\frac{1}{2} \varepsilon_3^{-1}\right)R^{-1}(\bw_h,\bw_h)+\left(1-\frac{1}{2}\varepsilon_3 \overline{\beta}_{D,d}^{-2}\right)
R \Vert (p_h,\hp_h)\Vert^2_{\text{HDG}}. \nonumber
\end{align}
\end{subequations}
By choosing $\varepsilon_1=\frac{1}{2}C_a^{-1}\beta_{S,d}^2$,
$\varepsilon_2 = \frac{1}{2}$, $\varepsilon_3=\beta_{D,d}^2$ the last
inequality becomes
\begin{subequations}
\allowdisplaybreaks
\begin{align}
\ooA_h((\bu_h,\hu_h),\bw_h,(p_h,\hp_h)) \ge& 
(\delta-C_a \beta_{S,d}^{-2})a_h^{\text{HDG}}((\bu_h,\hu_h),
(\bu_h,\hu_h))+ (\delta-1)\lambda (\divv \bu_h,\divv \bu_h) \nonumber \\
%&  \nonumber \\
& + \left( \delta S+\frac{1}{\lambda_0}-\frac{1}{4} \frac{1}{\lambda_0}
-\frac{1}{4} \frac{1}{\lambda_0}\right) (p_h,p_h) \nonumber \\
& + \left(\delta-\frac{1}{2}\right)R^{-1}(\bw_h,\bw_h) 
 +\left(1-\frac{1}{2}\right) R \Vert (p_h,\hp_h)\Vert^2_{\text{HDG}}.
\nonumber 
\end{align}
\end{subequations}
For $\delta\ge\max\left\{\frac{3}{2},\frac{1}{2 C_c}+C_a \overline{\beta}_{S,d}^{-2}\right\}$, we finally obtain
\begin{subequations}
\allowdisplaybreaks
\begin{align}
\ooA_h((\bu_h,\hu_h),\bw_h,(p_h,\hp_h)) \ge& 
\frac{1}{2}
\left(
\Vert(\bu_h,\hu_h)\Vert_{\text{HDG}}^2
+\lambda \Vert \divv \bu_h\Vert^2 
\right. 
\nonumber \\
&
 + \left(  S+\frac{1}{\lambda_0}\right) \Vert p_h\Vert_0^2
+ R\Vert (p_h,\hp_h)\Vert_{\text{HDG}}^2   + \left. R^{-1}\Vert \bw_h \Vert_0^2 \right)
 \nonumber \\
 = & \frac{1}{2}\left( \Vert (\bu_h,\hu_h)\Vert_{\overline{\bU}_h}^2 + \Vert (p_h,\hp_h)\Vert_{\overline{P}_h}^2 + \Vert \bw_h\Vert_{\bW^{-}}^2 \right),
\nonumber 
%\\
%\ge & \frac{1}{6}\left(  \Vert (\bu_h,\hu_h)\Vert_{\overline{\bU}_h} + \Vert (p_h,\hp_h)\Vert_{\overline{P}_h} + \Vert \bw_h\Vert_{\bW^{-}}  \right)^2,
%\nonumber
\end{align}
%where we have used that
utilizing $a_h^{\text{HDG}}((\bu_h,\hu_h))\ge C_c \Vert (\bu_h,\hu_h) \Vert_{\text{HDG}}^2$.
\end{subequations}
\end{proof}

\subsection{Uniform preconditioners}\label{sec::uniformprecond}
The results from the previous subsection imply a ``mapping property'' that is the basis for defining uniform
preconditioners. Here, we discuss norm-equivalent (block-diagonal) preconditioners which fall into this 
category. 
%and comment on Field-of-Values (FoV)-equivalent (block-triangular) preconditioners which fall into this category.
%
%\subsubsection{Well-posedness and preconditioning}

Consider a uniformly well-posed problem of the form~\eqref{var_2_3} where $\mathcal{A}: \bm X\rightarrow  {\bm X}^{\prime}$ 
is a linear operator, i.e.,
$\mathcal{A}\in \mathcal{L}({\bm X}, {\bm X}^{\prime})$, 
$\mathcal{F}\in {\bm X}^{\prime}$ for a given Hilbert space $\bm X$, e.g., $\bm X:=\bm U\times \bm W\times \bm P$ or 
$\bm X:=\overline{\bm X}_h=\overline{\bm U}_h\times \bm W_h\times P_h$, or 
$\bm X:=\ooX_h=\overline{\bm U}_h\times \bm W_h\times \overline{P}_h$. Here we assume that 
$\mathcal{A}$ and $\mathcal{F}$ are defined 
via the bilinear and linear forms $\mathcal{A}(\cdot,\cdot)$, $\mathcal{F}(\cdot)$, or 
$\overline{\mathcal{A}}_h(\cdot,\cdot)$, $\overline{\mathcal{F}}_h(\cdot)$, or 
$\overline{\overline{\mathcal{A}}}_h(\cdot,\cdot)$, $\overline{\overline{\mathcal{F}}}_h(\cdot)$, 
cf.~\eqref{var_2_3}, \eqref{HDG_both}, \eqref{HDG-HM_both}.
%Uniform well-posedness, according to Definition 4.1, means that the constants $C$ and $\beta$ 
%in~\eqref{boundedness_A} and \eqref{inf_sup_A} do not depend on any of the problem parameters. 
Let us write equation~\eqref{var_2_3} in operator form, i.e., 
\begin{equation}\label{op_2_3}
\mathcal{A}\bx = \mathcal{F}\in  {\bm X}^{\prime}
\end{equation} 
and define the linear operator $\mathcal{B}:   {\bm X}^{\prime} \rightarrow \bm X$, i.e., 
$\mathcal{B}\in \mathcal{L}( {\bm X}^{\prime}, \bm X)$ by 
\begin{equation}\label{def_B}
(\mathcal{B}\mathcal{G},\by)_{\bm X}=\langle \mathcal{G},\bm y\rangle, \qquad 
\forall \mathcal{G}\in  {\bm X}^{\prime}, \by\in \bm X, 
\end{equation}
where $(\cdot,\cdot)_{\bm X}$ is the inner product inducing the norm $\Vert \cdot \Vert_{\bm X}$, that is, 
$\Vert \by \Vert_{\bm X}= (\by,\by)_{\bm X}^{\frac{1}{2}}$, or, equivalently, $\mathcal{B}^{-1}:{\bm X}\rightarrow 
{\bm X}^{\prime}$, $\mathcal{B}^{-1}\in\mathcal{L}({\bm X},{\bm X}^{\prime})$ by 
\begin{equation}
\langle \mathcal{B}^{-1}\bx,\by\rangle = (\bx,\by)_{\bm X}, \qquad \forall \bx,\by\in \bm X,
\end{equation} 
which implies
\begin{equation}
\langle \mathcal{B}^{-1}\bx,\bx\rangle = (\bx,\bx)_{\bm X}=\Vert \bx\Vert_{\bm X}^2, \qquad \forall \bx\in \bm X.
\end{equation} 
In practice, the latter relation is often replaced by the weaker condition 
\begin{equation} \label{eq::defprecond}
\langle \mathcal{B}^{-1}\bx,\bx\rangle \eqsim \Vert \bx\Vert_{\bm X}^2,
\end{equation}
for which reason the preconditioner $\mathcal{B}$ is also referred to as a norm-equivalent preconditioner, 
cf.~\cite{mardal2011preconditioning}. The symbol "$\eqsim$" stands for a norm equivalence, uniform with respect 
to all problem parameters.

Since~\eqref{boundedness_A} and \eqref{inf_sup_A} are in the norm $\Vert \cdot\Vert_{\bm X}$, we conclude 
for the operators $\mathcal{B}\mathcal{A}\in \mathcal{L}(\bm X,\bm X)$ and $(\mathcal{B}\mathcal{A})^{-1}\in 
\mathcal{L}(\bm X,\bm X)$ the following bounds: 
\begin{equation}\label{prec_op_bound}
\Vert \mathcal{B}\mathcal{A}\Vert_{\mathcal{L}(\bm X,\bm X)}=\sup_{\bx,\by}\frac{(\mathcal{B}\mathcal{A} \bx,\by)_{\bm X}}
{\Vert \bx\Vert_{\bm X}\Vert \by\Vert_{\bm X}} 
= \sup_{\bx,\by}\frac{\langle \mathcal{A}\bx,\by \rangle}{\Vert \bx\Vert_{\bm X}\Vert \by\Vert_{\bm X}}
=   \sup_{\bx,\by}\frac{ \mathcal{A}(\bx,\by)}{\Vert \bx\Vert_{\bm X}\Vert \by\Vert_{\bm X}}\le C,
\end{equation}
\begin{align}\label{prec_op_inf_sup}
\left(\Vert (\mathcal{B}\mathcal{A})^{-1}\Vert_{\mathcal{L}(\bm X,\bm X)}\right)^{-1}
=& \inf_{\bx} \left(  \frac{1}{\displaystyle\sup_{\by} \frac{((\mathcal{B}\mathcal{A})^{-1}\bx,\by)_{\bm X}}
{\Vert \bx\Vert_{\bm X}\Vert \by\Vert_{\bm X}} }\right)
=
\inf_{\bx}\sup_{\by}\frac{(\mathcal{B}\mathcal{A} \bx,\by)_{\bm X}}
{\Vert \bx\Vert_{\bm X}\Vert \by\Vert_{\bm X}} 
\nonumber \\
=& \inf_{\bx}\sup_{\by}\frac{\langle \mathcal{A}\bx,\by \rangle}{\Vert \bx\Vert_{\bm X}\Vert \by\Vert_{\bm X}} 
= \inf_{\bx}  \sup_{\by}\frac{ \mathcal{A}(\bx,\by)}{\Vert \bx\Vert_{\bm X}\Vert \by\Vert_{\bm X}}\ge \beta.
\end{align}

Finally, \eqref{prec_op_bound} and \eqref{prec_op_inf_sup} together imply that the condition number $\kappa$ of the preconditioned 
operator $\mathcal{B}\mathcal{A}\in \mathcal{L}(\bm X,\bm X)$ is uniformly bounded by a constant that does not depend on any 
problem parameters, i.e.,
\begin{equation}\label{cond_est}
\kappa(\mathcal{B}\mathcal{A}):=\Vert \mathcal{B}\mathcal{A} \Vert_{\mathcal{L}(\bm X,\bm X)} 
\Vert (\mathcal{B}\mathcal{A})^{-1} \Vert_{\mathcal{L}(\bm X,\bm X)} \le \frac{C}{\beta}.
\end{equation}

\subsection{Optimal error estimates}

The uniform well-posedness that we have established in
Theorem~\ref{well-posed_hdg-hm} for the hybridized/hybrid mixed
discretization implies near best approximation estimates, which we
state next. For the following statements let $(\bu,\bw,p)$ be the
exact solution of the continuous problem~\eqref{eqn:weak-3f} assuming
that
\begin{align} \label{eq::exactreg}
  \bu \in \bm{H}^1_0(\Omega) \cap \bm{H}^2(\mathcal{T}_h), \quad 
  \bw \in \bm{H}_0(\div,\Omega),\quad \textrm{and} \quad 
  p \in H^1(\Omega) \cap H^2(\mathcal{T}_h),
\end{align}
where
$\bm{H}^m(\mathcal{T}_h) := \{v \in \bm{L}^2(\Omega): v|_T \in \bm{H}^m(T)~\forall T \in
\mathcal{T}_h \}$ is the broken Sobolev space of order $m$. Further
let $\overline{\bu} := (\bu,\hat \bu)$ and
$\overline{p} := (p, \hat p)$ with $\hat \bu := \bu|_{\mathcal{F}_h}$
and $\hat p := p|_{\mathcal{F}_h}$.
\begin{theorem}\label{thm:error}
Consider problem~\eqref{HDG-HM_gen}--\eqref{HDG-HM_both} 
as a discretization of the continuous problem~\eqref{eqn:weak-3f} 
in three-field formulation and assume that the exact solution fulfills \eqref{eq::exactreg}. Then the following near-best approximation result holds with a constants 
$\overline{\overline{C}}_{uv},\overline{\overline{C}}_{p} $ independent 
of all problem parameters:
\begin{align}
  \Vert \overline{\bu}-\overline{\bu}_h \Vert_{\overline{\bU}_h}+\Vert \bw-\bw_h \Vert_{\bW^{-}}
  \le \overline{\overline{C}}_{uv}
    \left(\inf_{\overline{\bv}_h\in \overline{\bU}_h}\Vert \overline{\bu} -\overline{\bv}_h 
    \Vert_{\overline{\bU}_h}+\inf_{\bz_h\in \bW^{-}_h}\Vert \bw-\bz_h\Vert_{\bW^{-}} \right) \label{error_est_u}\\
  \Vert \overline{p}-\overline{p}_h \Vert_{\overline{P}_h} 
  \le \overline{\overline{C}}_p
    \left(\inf_{\overline{\bv}_h\in \overline{\bU}_h}\Vert \overline{\bu} -\overline{\bv}_h 
    \Vert_{\overline{\bU}_h}+\inf_{\bz_h\in \bW^{-}_h}\Vert \bw-\bz_h\Vert_{\bW^{-}} +\inf_{\overline{q}_h\in \overline{P}_h}
    \Vert \overline{p}-\overline{q}_h\Vert_{\overline{P}_h} \right) \label{error_est_p}
\end{align}
\end{theorem}

\begin{proof}
The proof follows the lines of the proof of Theorem 5.2 in~\cite{HongKraus2018}.
\end{proof}

\begin{remark}
An analogous result to Theorem~\ref{thm:error} is also valid for the discrete problem~\eqref{HDG_gen}--\eqref{HDG_both} 
if one replaces the spaces $\bW^{-}$, $\bW_h^{-}$, $\overline{P}_h$ by $\bW$, $\bW_h$ and $P_h$ and the corresponding 
norms
 $\Vert \cdot \Vert_{\bW^{-}}$ and $\Vert \cdot \Vert_{\overline{P}_h}$ by $\Vert \cdot \Vert_{\bW}$ and $\Vert \cdot \Vert_P$. 
 The result is then a consequence of Theorem~\ref{well-posed_hdg}.
\end{remark}

In the following let
$\overline{\Pi}_{P_h}(\cdot) =( \Pi_{P_h}(\cdot), \Pi_{\hat
  P_h}(\cdot)) \in \overline{P}_h$ be the standard element and
facet-wise $L^2$-projection. Using the proper, well known (see
\cite{Boffi2013mixed, arnold2002unified, LehrenfeldSchoeberl2016high})
interpolation operators and standard arguments, one can derive the
following optimal error estimates from the above best approximation
results.
\begin{theorem} \label{th::errorest} Consider
  problem~\eqref{HDG-HM_gen}--\eqref{HDG-HM_both} as a discretization
  of the continuous problem~\eqref{eqn:weak-3f} in three-field
  formulation. Beside \eqref{eq::exactreg} we assume that the exact
  solution fulfills the regularity estimate
  $(\bu, \bw, p) \in \bm{H}^{m}(\mathcal{T}_h) \times
  \bm{H}^{m-1}(\mathcal{T}_h) \times {H}^{m-1}(\mathcal{T}_h)$. Then
  there hold the following error estimates with a constants
  $\overline{\overline{C}}_{e,uv},\overline{\overline{C}}_{e,p} $
  independent of all problem parameters:
  \begin{align*}
    \Vert \overline{\bu}-\overline{\bu}_h \Vert_{\overline{\bU}_h}+\Vert \bw-\bw_h \Vert_{\bW^{-}}
    &+\Vert \overline{\Pi}_{P_h}\overline{p}-\overline{p}_h \Vert_{\overline{P}_h} \\
    &\le \overline{\overline{C}}_{e, uv}
      h^{s} ( | \bu |_{\bm{H}^{s+1}(\mathcal{T}_h)} + \lambda^{\frac{1}{2}} |\div(\bu)|_{\bm{H}^{s}(\mathcal{T}_h)} + R^{-\frac{1}{2}}| \bw |_{\bm{H}^{s}(\mathcal{T}_h)}),
  \end{align*}
\begin{align*}  
    \Vert \overline{p}-\overline{p}_h \Vert_{\overline{P}_h} 
    &\le \overline{\overline{C}}_{e,p}
      h^{s-1} ( | \bu |_{\bm{H}^{s}(\mathcal{T}_h)} + \lambda^{\frac{1}{2}} |\div(\bu)|_{\bm{H}^{s-1}(\mathcal{T}_h)} \\
    &\qquad \qquad \qquad\qquad+ R^{-\frac{1}{2}}| \bw |_{\bm{H}^{s-1}(\mathcal{T}_h)} + R^{\frac{1}{2}} | p |_{H^{s}(\mathcal{T}_h)} + \gamma^{\frac{1}{2}} | p |_{H^{s-1}(\mathcal{T}_h)}).
  \end{align*}
  where $s := \min\{l,m-1\}.$
\end{theorem}
%\subsubsection{Parameter-robust preconditioners for the time-discrete problem}

\begin{remark}
  Assuming enough regularity of the exact solution, Theorem
  \ref{th::errorest} shows that the projected error
  $\Vert \overline{\Pi}_{P_h}\overline{p}-\overline{p}_h
  \Vert_{\overline{P}_h}$ converges with one order higher than
  $\Vert \overline{p}-\overline{p}_h \Vert_{\overline{P}_h}$. This
  super convergence property of (hybrid) mixed methods is well known
  in the literature, see for example \cite{stenberg2011,MR2860674}.
\end{remark}
\subsection{Implementation aspects and static condensation} 
In order to solve the discrete system, we employ static
condensation of the local degrees of freedom. These are given by the
dof introduced through the discontinuous approximation spaces $\bW_h^{-}$
and $P_h$. 
%Further, 
One can also eliminate the local $\bm H(\div)$-conforming
element bubbles of the space $\bU_h$. However, for ease of
representation, we only consider the lowest order case $l =1$, hence, no
bubbles for the displacement are present. In the following, we use the
same symbols $\overline \bu_h := (\bu_h, \hat \bu_h)$, $\bw_h$, $p_h$
and $\hat p_h$ for the representation of the coefficients of the
corresponding discrete finite element solutions.  Then
\eqref{Cons-disc-3-field-HDG-HM} can be written as
\begin{align*}
  \begin{pmatrix}
    A_{\overline u} & 0 &B_u^\T & 0 \\
    0 & M_w & B_w^\T & \hat B_w^\T \\
    B_u & B_w & -M_p & 0 \\
    0 & \hat B_w & 0 & 0\\
  \end{pmatrix}
  \begin{pmatrix}
    \overline \bu_h\\
    \bw_h \\
    p_h \\
    \hat p_h\\
  \end{pmatrix} =
  \begin{pmatrix}
    \bf_h\\ \bm 0\\g_h\\0
  \end{pmatrix},
\end{align*}
where $\bf_h$ represent the corresponding vector of the right hand
side $(\bf, \bv_h)$ and $g_h$ the vector of $(g, q_h)$. Further,
$A_{\overline u},B_u, M_w,M_p, B_w$ and $\hat B_w$ denote the
operators, or their corresponding matrix representations, defined via
the bilinear forms
$a_{h}((\bu_h, \hat{\bu}_h)$, $(\bv_h, \hat{\bv}_h))$, $(-\div \bu_h, q_h)$,
$(R^{-1}w_h,z_h)$, $(Sp_h,q_h)$, $b((q_h,0), \bw_h)$ and
$b((0,\hat q_h), \bw_h) $, respectively. From the second line we see
that we can eliminate $\bw_h$
%by
using 
$
  \bw_h = M_w^{-1}(-B_w^\T p_h - \hat B_w^\T \hat p_h).
$
Then the third line gives 
$
  p_h = -(M_p + B_w M_w^{-1} B_w^\T)^{-1} (-B_u \overline \bu_h + B_w M_w^{-1} \hat B_w^\T \hat p_h).
$
Thus, 
%in total 
we have the following system to solve
\begin{align} \label{eq::condsys}
  \begin{pmatrix}
    A & B^\T\\B & -C
  \end{pmatrix}
  \begin{pmatrix}
    \overline \bu_h \\ \hat p_h\\
  \end{pmatrix} = 
  \begin{pmatrix}
    \bf_h \\ g_h\\
  \end{pmatrix},
\end{align}
with
\begin{align*}
  A &:=  A_{\overline u} + B_u^\T (M_p + B_w M_w^{-1} B_w^\T)^{-1} B_u, \\
  B &:= -\hat B_w M_w^{-1} B_w^\T (M_p + B_w M_w^{-1} B_w^\T)^{-1} B_u,\\
  C &:= \hat B_w M_w^{-1} B_w^\T (M_p + B_w M_w^{-1} B_w^\T)^{-1} B_w M_w^{-1} \hat B_w^\T + \hat B_w M_w^{-1} \hat B_w^\T.
\end{align*}
Note that $M_w,M_p$ and $(M_p + B_w M_w^{-1} B_w^\T)$ are all block
diagonal, thus locally invertible. Further, the latter operator is equivalent to
%reads as a matrix representation of 
a (scaled) $H^1$-like norm on $\hat P_h$.  By means of norm equivalent
preconditioning, cf. equation \eqref{eq::defprecond}, we now follow
two different approaches. The first preconditioner we investigate is
based on a block system that decouples mechanics from the flow
problem, and, additionally, the velocity from the fluid pressure. The
latter is achieved by introducing an HDG bilinear form on
$\overline P_h$ for the discretization of $\div(R \nabla p)$ as given
in the original equation \eqref{eq::twofield} (where $K$ was replaced
due to scaling by $R$). 
%The first preconditioner we investigate is based on the
%$H^1$-like HDG norm defined on $\overline P_h$ and relates directly to the presented 
%well-posedness analysis.
%To this end 
Henceforth, let $\tilde M_p$ denote the matrix
representation of the scaled bilinear form $(\gamma p_h,q_h)$. Then we
define the operator 
\begin{align*}
  \mathcal{B} :=
  \begin{pmatrix}
    A_{\overline u} & 0 &0 & 0 \\
    0 & M_w & 0 & 0 \\
    0 & 0 & -\tilde M_p - A_p & -B_p^\T \\
    0 & 0 & -B_p &  -A_{\hat p}\\
  \end{pmatrix}^{-1}.
\end{align*} 
%elasticity equation and the $H^1$-like system
%of the pressure. 
where $A_p$, $B_p$ and $A_{\hat p}$ correspond to the
bilinear forms given by
\begin{align*}
  a_p(p_h, q_h) &:= R \sum_{T \in \mathcal{T}_h}\int_T \nabla p_h \cdot \nabla q_h \dx + \int_{\partial T} (-\nabla p_h \cdot n q_h - \nabla q_h \cdot n p_h) + \eta_p{l^2}{h^{-1}} \eta_p p_h q_h \ds,\\
  b_p(p_h, \hat q_h) &:= R \sum_{T \in \mathcal{T}_h} \int_{\partial T} \nabla p_h \cdot n  \hat q_h -\eta_p{l^2}{h^{-1}}  p_h \hat q_h \ds, \\
  a_{\hat p}(\hat p_h, \hat q_h) &:= R \sum_{T \in \mathcal{T}_h} \int_{\partial T}\eta_p{l^2}{h^{-1}} \hat  p_h \hat q_h \ds,
\end{align*}
respectively, where $\eta_p$ is again a sufficiently large
stabilization parameter. Note that the combined bilinear form
$ a_p(p_h, q_h) + b_p(p_h, \hat q_h) + b_p(q_h, \hat p_h) + a_{\hat
  p}(\hat p_h, \hat q_h)$ is the HDG bilinear form mentioned above
which is continuous and elliptic with respect to
$R \Vert \cdot \Vert_{\text{HDG}}$.  Similarly, as before, we can
eliminate the local variables to obtain the following preconditioner
\begin{align} \label{eq::precondone}
  \begin{pmatrix}
     A_{\overline u} & 0\\0 & -(A_{\hat p} + B_p (\tilde M_p^{-1} + A_p^{-1}) B_p^T)\\
  \end{pmatrix}
\end{align}
for the condensed system \eqref{eq::condsys}, where we have again made
use of $A_p$ being block diagonal and invertible. Further,
note that both blocks on the diagonal are $H^1$-type systems. Thus,
standard solvers, such as, for example, an algebraic multigrid method
for the lowest order system and a ``balancing domain decomposition
with constriants''(BDDC) preconditioner, the latter featuring
robustness in the polynomial degree, can be used.

The second block diagonal preconditioner we test still satisfies the
norm equivalence \eqref{eq::defprecond}, but decouples only the
mechanics and flow problems, hence, keeps the hybrid mixed formulation
of the velocity pressure system.
% constructed on the same basic principle, however, now
%following the same basic ideas thereby
%exploiting the hybrid mixed
%formulation of the flow subsystem
The block diagonal operator preconditioner is then given by
\begin{align*}
  \mathcal{B} : = 
  \begin{pmatrix}
    A_{\overline u} & 0 &0 & 0 \\
    0 & M_w & B_w^\T & \hat B_w^\T \\
    0 & B_w & -\tilde M_p & 0 \\
    0 & \hat B_w & 0 & 0\\
  \end{pmatrix}^{-1}.
\end{align*}
Following similar steps as above, the preconditioner for the condensed
system is
\begin{align} \label{eq::precondtwo}
  \begin{pmatrix}
     A_{\overline u} & 0\\0 & -\tilde C
  \end{pmatrix},
\end{align}
where $\tilde C$ is the same as $C$ with $M_p$ replaced by
$\tilde M_p$. The advantage of the preconditioner defined by
\eqref{eq::precondtwo}, as demonstrated below in Section~\ref{sec::paramrob}, is that the subsystem for the pressure variable
does not require a stabilization parameter $\eta_p$ which in general
affects the condition number.

% \subsubsection{Hybridized DG method}

%\subsubsection{Hybridized DG/mixed hybrid methods}

%\subsection{The discrete case}
%
%The results from the previos two subsections carry over to the discrete case. We prove discrete well-posedness
%and formulate the uniform preconditioners for the finite-dimensional problems resulting from application of the new
%family of hybridized DG/mixed hybrid methods for the Biot problem.

\section{Numerical results}\label{numres}

In this section, we present several numerical examples to validate our theoretical findings. 
%In the first section 
First, we test for the expected orders of convergence for a problem with a constructed solution increasing the degree of the FE 
%polynomial 
approximation. 
%In the second example, 
Second, we study the parameter-robustness of the proposed preconditioners. 
Finally, we 
%finish this section with a discussion on 
discuss the cost efficiency of our modified methods. 
All numerical examples are implemented within the finite element library Netgen/NGSolve, see~\cite{netgen, ngsolve} and \url{www.ngsolve.org}.
\subsection{Convergence of the hybridized/hybrid mixed method}
Here we discuss the convergence orders of the errors of the methods introduced in this work. 
Note, however, that we only consider the discretization given by~\eqref{Cons-disc-3-field-HDG-HM} 
since the solution is the same as of~\eqref{Cons-disc-3-field-HDG}. 

\subsubsection{2D example} \label{sec::twodexample}
We solve problem \eqref{Cons-disc-3-field-HDG-HM} on the spatial domain $\Omega = (0,1)^2$ 
and choose the right hand side $\bm f$ and $g$ such that the exact solutions are given by
\begin{align*}
  \bm u &:= (-\partial_y \phi, \partial_x \phi), \quad 
          p := \sin(\pi x) \sin(\pi y) - p_0,
\end{align*}
with the potential $\phi = x^2(1-x)^2y^2(1-y)^2$ and
$p_0\in \mathbb{R}$ is chosen such that $p \in L^2_0(\Omega)$.  For
simplicity, we choose the constants $K = 1$, $\mu = 1$, $S_0 = 1$, and
$\alpha = 1$. Further, we set $\lambda = c$ with an arbitrary constant
$c \in \mathbb{R}^{+}$ since the exact and discrete solutions are exactly
divergence-free.
%Note that although the displacement is such that
%$\operatorname{div} \bm u =0$, there is a coupling between the
%equations due to $\nabla p$, see the first equation in
%\eqref{Cons-disc-3-field-HDG-HM}.
%Further note that the pressure has
%non-homogeneous Dirichlet boundary conditions.

In Table~\ref{twodexample} we have displayed several discrete errors
and their estimated order of convergence (eoc) for the discretization
of problem \eqref{Cons-disc-3-field-HDG-HM} for varying polynomial
orders $l = 1,2,3,4$. Whereas the $H^1$-seminorm error of the
displacement $\bm u_h$ and the pressure $\bm p_h$ converge with the
expected (see Theorem \ref{th::errorest}) order $\mathcal{O}(h^l)$ and
$\mathcal{O}(h^{l-1})$, respectively, the corresponding $L^2$-norm
errors converge with order $\mathcal{O}(h^{l+1})$ and
$\mathcal{O}(h^{l})$.  This can be shown by a standard Aubin-Nitsche
duality argument whenever the considered problem is sufficiently
regular, see for example \cite{Boffi2013mixed}. Note also that the
$L^2$-norm error $||\nabla p + R^{-1}\bm w_h||_0$ of the discrete
velocity $\bm w_h$ converges with optimal order
$\mathcal{O}(h^{l})$. In the lowest order case where we have a
piece-wise constant approximation of the pressure $p_h$, we do not
present the $H^1$-semi norm error of the pressure since the gradient
$\nabla p_h$ vanishes locally on each element.

\begin{table}[h]
\begin{center}
\scriptsize
%\begin{tabular}{r@{~}|@{~}c@{(}c@{)~}c@{(}c@{)}c@{(}c@{)~}c@{(}c@{)}}
%  \begin{tabular}{@{~}c@{~}|@{~~~}c@{~~~(}c@{)~~~}c@{~~~(}c@{)~~~}c@{~~~(}c@{)~~~}c@{~~~(}c@{)}}
%\begin{tabular}{@{~}c@{~}|c@{~(}c@{)~}cccccc}
  \begin{tabular}{r@{~}|@{~} c @{\hskip 0.1in(} c @{)\hskip 0.1in } c @{\hskip 0.1in(} c @{)\hskip 0.1in}c@{\hskip 0.1in(}c@{)\hskip 0.1in}c@{\hskip 0.1in(}c@{)\hskip 0.1in}c@{\hskip 0.1in(}c@{)}}
      \toprule
  $|\mathcal{T}|$ & $|| \nabla u - \nabla u_h||_0$ & \footnotesize eoc &$|| u - u_h||_0$ & \footnotesize eoc &$|| \nabla p - \nabla p_h ||_0$ & \footnotesize eoc &$|| p - p_h||_0$ & \footnotesize eoc & $||\nabla p + R^{-1}\bm w_h||_0$ & \footnotesize eoc   \\
\midrule
\multicolumn{11}{c}{$l=1$}\\
6& \num{0.05296588089681155}&--& \num{0.0054696012333586095}&--& -- &--& \num{0.18020014360177114}&--& \num{0.9739259104175237}&--\\
24& \num{0.04879615037983135}&\numeoc{0.12}& \num{0.004872874101873878}&\numeoc{0.17}& --&--& \num{0.17793880884375782}&\numeoc{0.02}& \num{0.5712239194201587}&\numeoc{0.77}\\
96& \num{0.022474791428348625}&\numeoc{1.12}& \num{0.001083464627642091}&\numeoc{2.17}& --&--& \num{0.08147932878196795}&\numeoc{1.13}& \num{0.2816014026231343}&\numeoc{1.02}\\
384& \num{0.010910873608265786}&\numeoc{1.04}& \num{0.0002569802984463594}&\numeoc{2.08}& --&--& \num{0.040226161431177666}&\numeoc{1.02}& \num{0.14204818433069816}&\numeoc{0.99}\\
1536& \num{0.005420984383887481}&\numeoc{1.01}& \num{6.294458285613617e-05}&\numeoc{2.03}& --&--& \num{0.020057644826557617}&\numeoc{1.00}& \num{0.07125695739500504}&\numeoc{1.00}\\
6144& \num{0.002707489300540879}&\numeoc{1.00}& \num{1.559755800239475e-05}&\numeoc{2.01}& --&--& \num{0.010022130873713384}&\numeoc{1.00}& \num{0.03566051564470094}&\numeoc{1.00}\\
\midrule
\multicolumn{11}{c}{$l=2$}\\
6& \num{0.045710841873620324}&--& \num{0.005065360004848349}&--& \num{1.6731380565552127}&--& \num{0.1532603117701809}&--& \num{0.29026631270561176}&--\\
24& \num{0.010630832938228198}&\numeoc{2.10}& \num{0.0005898897247898164}&\numeoc{3.10}& \num{0.7436905719424244}&\numeoc{1.17}& \num{0.0312242020673273}&\numeoc{2.30}& \num{0.0816872906157874}&\numeoc{1.83}\\
96& \num{0.0029997582309984434}&\numeoc{1.83}& \num{7.395586463234664e-05}&\numeoc{3.00}& \num{0.38010541290354677}&\numeoc{0.97}& \num{0.00786538228627508}&\numeoc{1.99}& \num{0.026810519576291648}&\numeoc{1.61}\\
384& \num{0.0007723988108188677}&\numeoc{1.96}& \num{9.276501869592985e-06}&\numeoc{3.00}& \num{0.19088995256566327}&\numeoc{0.99}& \num{0.0019687287289485915}&\numeoc{2.00}& \num{0.006943827598469703}&\numeoc{1.95}\\
1536& \num{0.00019392294646265008}&\numeoc{1.99}& \num{1.1647832119221928e-06}&\numeoc{2.99}& \num{0.09554622361026065}&\numeoc{1.00}& \num{0.0004923247171103386}&\numeoc{2.00}& \num{0.0017533074407490457}&\numeoc{1.99}\\
6144& \num{4.851190030542302e-05}&\numeoc{2.00}& \num{1.460873979084967e-07}&\numeoc{3.00}& \num{0.047785629152143724}&\numeoc{1.00}& \num{0.00012308993113571305}&\numeoc{2.00}& \num{0.00043989032082625295}&\numeoc{1.99}\\
\midrule 
\multicolumn{11}{c}{$l=3$}\\
6& \num{0.00884374685274383}&--& \num{0.0005071216333523856}&--& \num{0.2729265199813085}&--& \num{0.006247648246626007}&--& \num{0.09687170687032122}&--\\
24& \num{0.002343807003847277}&\numeoc{1.92}& \num{6.740517549246063e-05}&\numeoc{2.91}& \num{0.1431623646306074}&\numeoc{0.93}& \num{0.0038775127435579384}&\numeoc{0.69}& \num{0.01010628196894606}&\numeoc{3.26}\\
96& \num{0.0003084217005805856}&\numeoc{2.93}& \num{4.3071933465229906e-06}&\numeoc{3.97}& \num{0.0365028045441106}&\numeoc{1.97}& \num{0.000456252975448396}&\numeoc{3.09}& \num{0.001344619138095218}&\numeoc{2.91}\\
384& \num{3.711178333778603e-05}&\numeoc{3.05}& \num{2.5425474698110053e-07}&\numeoc{4.08}& \num{0.009172023703079351}&\numeoc{1.99}& \num{5.662568335139901e-05}&\numeoc{3.01}& \num{0.00017099003665718373}&\numeoc{2.98}\\
1536& \num{4.575986131065168e-06}&\numeoc{3.02}& \num{1.5567036300243737e-08}&\numeoc{4.03}& \num{0.0022958922252487306}&\numeoc{2.00}& \num{7.0680159247630146e-06}&\numeoc{3.00}& \num{2.1511777914141894e-05}&\numeoc{2.99}\\
6144& \num{5.690948303541737e-07}&\numeoc{3.01}& \num{9.649010314193874e-10}&\numeoc{4.01}& \num{0.0005741533845698664}&\numeoc{2.00}& \num{8.83200788376908e-07}&\numeoc{3.00}& \num{2.6960400043806163e-06}&\numeoc{3.00}\\
\midrule
\multicolumn{11}{c}{$l=4$}\\
6& \num{0.00327936256816333}&--& \num{0.0002644323121791122}&--& \num{0.19771413643894445}&--& \num{0.002342321636437553}&--& \num{0.010429028216310096}&--\\
24& \num{0.0002773400550254285}&\numeoc{3.56}& \num{1.0102477116344592e-05}&\numeoc{4.71}& \num{0.019392914068972145}&\numeoc{3.35}& \num{0.0001133247448091654}&\numeoc{4.37}& \num{0.0008265915589515648}&\numeoc{3.66}\\
96& \num{1.578211797530256e-05}&\numeoc{4.14}& \num{2.8802671891463547e-07}&\numeoc{5.13}& \num{0.0024702267449469902}&\numeoc{2.97}& \num{7.376555610602733e-06}&\numeoc{3.94}& \num{5.2690191060316845e-05}&\numeoc{3.97}\\
384& \num{9.836857868045258e-07}&\numeoc{4.00}& \num{8.947229023753106e-09}&\numeoc{5.01}& \num{0.000310122806640755}&\numeoc{2.99}& \num{4.6645188910115285e-07}&\numeoc{3.98}& \num{3.3199976597046184e-06}&\numeoc{3.99}\\
1536& \num{6.149564581864992e-08}&\numeoc{4.00}& \num{2.795790787499885e-10}&\numeoc{5.00}& \num{3.880729080970399e-05}&\numeoc{3.00}& \num{2.9238955667268187e-08}&\numeoc{4.00}& \num{2.0850062120032512e-07}&\numeoc{3.99}\\
6144& \num{3.842815155171198e-09}&\numeoc{4.00}& \num{8.736007252103972e-12}&\numeoc{5.00}& \num{4.852219815146468e-06}&\numeoc{3.00}& \num{1.828783630627499e-09}&\numeoc{4.00}& \num{1.3066147725719299e-08}&\numeoc{4.00}\\
\bottomrule
\end{tabular}
\caption{ The $H^1$-seminorm and the $L^2$-norm errors of the discrete
  displacement $\bm u_h$ and the discrete pressure $p_h$ and the
  $L^2$-norm errors of the discrete velocity $\bm w_h$ for different
  polynomial degrees $l=1,2,3,4$ for the two-dimensional
  example}\label{twodexample}
\end{center}
\end{table}

\subsubsection{3D example}
We solve problem \eqref{Cons-disc-3-field-HDG-HM} on the spatial domain $\Omega = (0,1)^3$ and 
choose the right hand side $\bm f$ and $g$ such that the exact solutions are given by
\begin{align*}
  \bm u &:= \operatorname{curl}(\phi, \phi, \phi), \quad
          p := \sin(\pi x) \sin(\pi y) \sin(\pi z)- p_0,
\end{align*}
with the potential $\phi = x^2(1-x)^2y^2(1-y)^2z^2(1-z)^2$ and
$p_0\in \mathbb{R}$ is chosen such that $p \in L^2_0(\Omega)$. The
parameters $\lambda, \mu, S_0, \alpha, K$ are chosen as in the
two-dimensional example.

%For
%simplicity, we choose the constants $\lambda = 0$, $\mu = 1$,
%$S_0 = 1$,$\alpha = 1$ and $K = 1$.

Again, we present in Table \ref{threedexample} several discrete errors
and their estimated orders of convergence for varying polynomial
degree $l= 1,2,3$. We make the same observations as for the
two-dimensional example, that is, all errors converge with optimal
order as predicted by Theorem \ref{th::errorest}.

\begin{table}[h]
\begin{center}
\scriptsize
%\begin{tabular}{r@{~}|@{~}c@{(}c@{)~}c@{(}c@{)}c@{(}c@{)~}c@{(}c@{)}}
%  \begin{tabular}{@{~}c@{~}|@{~~~}c@{~~~(}c@{)~~~}c@{~~~(}c@{)~~~}c@{~~~(}c@{)~~~}c@{~~~(}c@{)}}
%\begin{tabular}{@{~}c@{~}|c@{~(}c@{)~}cccccc}
  \begin{tabular}{r@{~}|@{~} c @{\hskip 0.1in(} c @{)\hskip 0.1in } c @{\hskip 0.1in(} c @{)\hskip 0.1in}c@{\hskip 0.1in(}c@{)\hskip 0.1in}c@{\hskip 0.1in(}c@{)\hskip 0.1in}c@{\hskip 0.1in(}c@{)}}
      \toprule
  $|\mathcal{T}|$ & $|| \nabla u - \nabla u_h||_0$ & \footnotesize eoc &$|| u - u_h||_0$ & \footnotesize eoc &$|| \nabla p - \nabla p_h ||_0$ & \footnotesize eoc &$|| p - p_h||_0$ & \footnotesize eoc & $||\nabla p + R^{-1}\bm w_h||_0$ & \footnotesize eoc   \\
\midrule
\multicolumn{9}{c}{$l=1$}\\
48& \num{0.005227449771717743}&--& \num{0.00048743300287077134}&--& -- &--& \num{0.17875881635511925}&--& \num{0.938029629205696}&--\\
384& \num{0.0026306366398559894}&\numeoc{0.99}& \num{0.0001696907610001353}&\numeoc{1.52}& --&--& \num{0.09582572511825986}&\numeoc{0.90}& \num{0.4947651951144149}&\numeoc{0.92}\\
3072& \num{0.0013295024502456736}&\numeoc{0.98}& \num{5.234170577436774e-05}&\numeoc{1.70}& --&--& \num{0.04878907706777089}&\numeoc{0.97}& \num{0.25069964409994755}&\numeoc{0.98}\\
24576& \num{0.0006631515813347822}&\numeoc{1.00}& \num{1.4453638817230519e-05}&\numeoc{1.86}& --&--& \num{0.024506275864878867}&\numeoc{0.99}& \num{0.12577216958101642}&\numeoc{1.00}\\
\midrule
\multicolumn{9}{c}{$l=2$}\\
48& \num{0.004783888455125389}&--& \num{0.0004726292199702656}&--& \num{1.0725965368017918}&--& \num{0.06390616021392538}&--& \num{0.27644549113542194}&--\\
384& \num{0.000858065493241947}&\numeoc{2.48}& \num{4.222456711077921e-05}&\numeoc{3.48}& \num{0.5820959653139401}&\numeoc{0.88}& \num{0.017326655307679788}&\numeoc{1.88}& \num{0.0742414347456475}&\numeoc{1.90}\\
3072& \num{0.00019385851489781293}&\numeoc{2.15}& \num{3.5495123082639975e-06}&\numeoc{3.57}& \num{0.2974833013245846}&\numeoc{0.97}& \num{0.004421505259018461}&\numeoc{1.97}& \num{0.018976466208924396}&\numeoc{1.97}\\
24576& \num{4.617265221489976e-05}&\numeoc{2.07}& \num{3.361482582452319e-07}&\numeoc{3.40}& \num{0.14958210011137416}&\numeoc{0.99}& \num{0.0011110440786274447}&\numeoc{1.99}& \num{0.004782518494704592}&\numeoc{1.99}\\
\midrule
\multicolumn{9}{c}{$l=3$}\\
48& \num{0.001159698754824411}&--& \num{7.426135402134066e-05}&--& \num{0.4422936039359659}&--& \num{0.008542651836505722}&--& \num{0.05282309324861693}&--\\
384& \num{0.00012683028766367077}&\numeoc{3.19}& \num{3.6381075572297214e-06}&\numeoc{4.35}& \num{0.12205880996355477}&\numeoc{1.86}& \num{0.001040986870506061}&\numeoc{3.04}& \num{0.006785802752549436}&\numeoc{2.96}\\
3072& \num{1.5384395189638782e-05}&\numeoc{3.04}& \num{2.0881187389926046e-07}&\numeoc{4.12}& \num{0.03131641563860158}&\numeoc{1.96}& \num{0.00012830938518940084}&\numeoc{3.02}& \num{0.0008564980891275221}&\numeoc{2.99}\\
24576& \num{1.8812015427661355e-06}&\numeoc{3.03}& \num{1.2158359754320982e-08}&\numeoc{4.10}& \num{0.007881146860190147}&\numeoc{1.99}& \num{1.5967981613171428e-05}&\numeoc{3.01}& \num{0.000107820904318774}&\numeoc{2.99}\\
    \bottomrule
\end{tabular}
\caption{
  The $H^1$-seminorm and the $L^2$-norm errors of the discrete displacement $\bm u_h$ and the discrete pressure $p_h$ and the $L^2$-norm errors of the discrete velocity $\bm w_h$ for different polynomial degrees $l=1,2,3$ for the three-dimensional example}\label{threedexample}
\end{center}
\end{table}

\subsection{Parameter-robustness of the
  preconditioners} \label{sec::paramrob} In this section, we
demonstrate the robustness of the preconditioners defined in Section~\ref{sec::uniformprecond} with respect to varying physical parameters.
%To this aim 
Again, we solve 
%again 
the example
given in Section~\ref{sec::twodexample} on a fixed triangulation with
384 elements. 
%We solve 
The system is solved by means of the minimal residual method (MinRes)
with a fixed tolerance of $10^{-10}$ and for different polynomial
degrees $l = 1,2,3,4$. In Figure \ref{fig:extwo} we plot the number of
iterations for the preconditioner defined in \eqref{eq::precondone}
with a fixed stabilization parameter $\eta_p = 10$ for variations of
the parameters $R^{-1}, \lambda, S$. In Figure \ref{fig:exone} we plot
the number of iterations for the same example using the preconditioner
defined in \eqref{eq::precondtwo}. Although both preconditioners show
the expected robustness as predicted by the analysis presented in Section~\ref{sec::wellposed}, we see that the results with
\eqref{eq::precondtwo} demonstrate improvement upon those 
%are much better than 
with~\eqref{eq::precondone}.  Besides resulting in a smaller number of
iterations, the second preconditioner \eqref{eq::precondtwo} is
substantially more robust with respect to the polynomial degree
$l$. We emphasize that the definition of \eqref{eq::precondone}
includes a proper scaling of the interior penalty stabilization
parameter with respect to the polynomial order given by
$\mathcal{O}(l^2)$. Although a different (smaller) stabilization
parameter might lead to better results--we have fixed $\eta_p = 10$
here--the analysis unfortunately only shows that $\eta_p$ has to be
chosen sufficiently large (see \cite{arnold2002unified}), its optimal
choice is difficult. Therefore, it is obvious that the mixed
formulation, which is known to result in a minimal stabilization, as
used in \eqref{eq::precondtwo}, is preferable.

  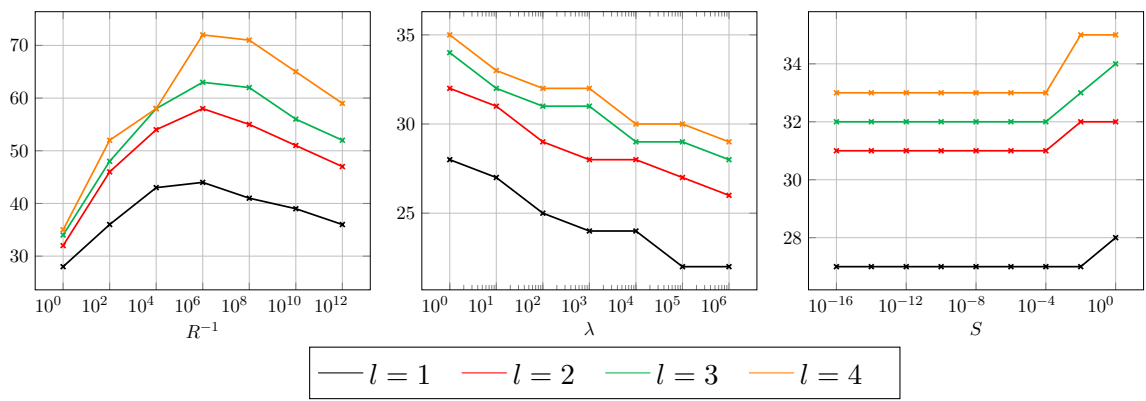
\begin{figure}[h]
  \begin{center}
  \pgfplotstableread{tabledata/Rinv_robustness_h1pc_minres.dat}       \rinv
\pgfplotstableread{tabledata/lam_robustness_h1pc_minres.dat}       \lam
\pgfplotstableread{tabledata/S_robustness_h1pc_minres.dat}       \S
                         
\begin{tikzpicture}
  [
  %spy using outlines={rounded rectangle, width=1.5cm, height=0.15cm, very thick,black,magnification=3, connect spies}, 
  scale=0.65
  ]
  \begin{groupplot}[
    group style ={group size = 3 by 1,horizontal sep=30pt,vertical sep=60pt},
    name=plot2,
    scale=1.0,
    legend columns=4,
    %legend rows=3,
    legend style={text width=3em,text height=0.9em },
    %ylabel=it,
    % ylabel=$U_B$,
    %xmin = 2e2,
    %xmax=2e5,
    xmode=log,
    %ymax=1e-1,
    %ymin=1e-5,
    %ymode=log,
    %ytick = {1e-5,1e-4,1e-3,1e-2,1e-1},
    %y tick label style={/pgf/number format/.cd,fixed,precision=2},
    %x tick label style={/pgf/number format/.cd,fixed,precision=2},
    yticklabel style={text width=3em,align=right},
    xticklabel style={text width=3em,align=left},
    % 
    % log basis x=2,
     grid=major,
    %% major grid style={black!2},
    %legend style={
    %  cells={align=left},
    %  % draw=none,
    %  at={(0.1,0.1)},
    %  anchor = north west
    %},
    % xticklabel=\pgfmathparse{2^\tick}\pgfmathprintnumber{\pgfmathresult}
    ]
        
    \nextgroupplot[legend to name={mylegend}, xlabel=$R^{-1}$, xtick = {1,1e2,1e4,1e6,1e8,1e10,1e12}]
    \addlegendentry{$l=1$}
    \addlegendimage{black}
    \addlegendentry{$l=2$}
    \addlegendimage{red}
    \addlegendentry{$l=3$}
    \addlegendimage{green!70!blue}
    \addlegendentry{$l=4$}
    \addlegendimage{red!50!yellow}
    
    \addplot[line width=1pt, mark=x, mark options={solid}, color=black] table[x=0,y=1]{\rinv};
    \addplot[line width=1pt, mark=x, mark options={solid}, color=red] table[x=0,y=2]{\rinv};
    \addplot[line width=1pt, mark=x, mark options={solid}, color=green!70!blue] table[x=0,y=3]{\rinv};
    \addplot[line width=1pt, mark=x, mark options={solid}, color=red!50!yellow] table[x=0,y=4]{\rinv};

    \nextgroupplot[legend to name=dummy, xlabel = $\lambda$, xtick = {1,1e1,1e2,1e3,1e4,1e5,1e6}]
    \addplot[line width=1pt,mark=x, mark options={solid}, color=black] table[x=0,y=1]{\lam};
    \addplot[line width=1pt,mark=x, mark options={solid}, color=red] table[x=0,y=2]{\lam};
    \addplot[line width=1pt,mark=x, mark options={solid}, color=green!70!blue] table[x=0,y=3]{\lam};
    \addplot[line width=1pt,mark=x, mark options={solid}, color=red!50!yellow] table[x=0,y=4]{\lam};

    \nextgroupplot[legend to name=dummy, xlabel = $S$, xtick = {1e-16,1e-12,1e-8,1e-4,1}]
    \addplot[line width=1pt, mark=x, mark options={solid}, color=black] table[x=0,y=1]{\S};
    \addplot[line width=1pt, mark=x, mark options={solid}, color=red] table[x=0,y=2]{\S};
    \addplot[line width=1pt, mark=x, mark options={solid}, color=green!70!blue] table[x=0,y=3]{\S};
    \addplot[line width=1pt, mark=x, mark options={solid}, color=red!50!yellow] table[x=0,y=4]{\S};
       
  \end{groupplot}

  \node[right=1em,inner sep=0pt] at (5,-1.7) {\pgfplotslegendfromname{mylegend}};

\end{tikzpicture}
 
%%% Local Variables:
%%% mode: latex
%%% TeX-master:"../main"
%%% End:
  \vspace{5pt}
  \caption{Robustness of the preconditioner defined in \eqref{eq::precondone}} \label{fig:extwo}
  \end{center}
\end{figure}

  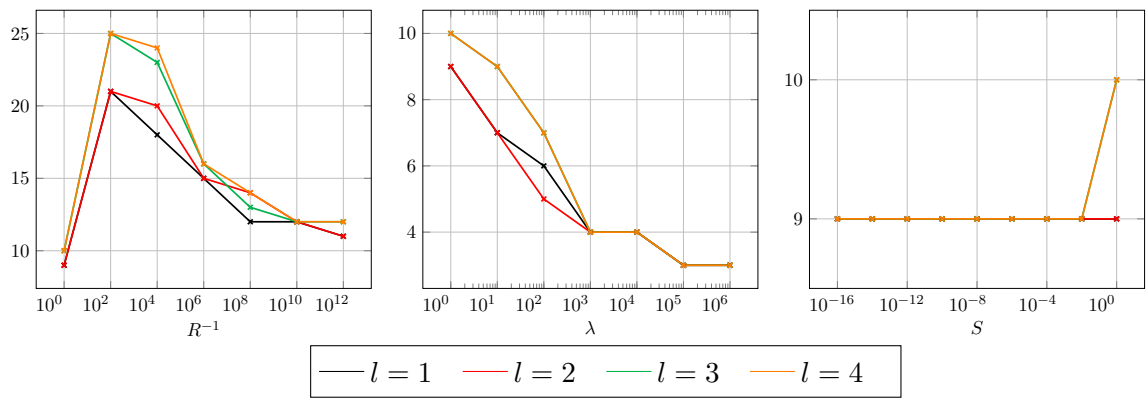
\begin{figure}[h]
  \begin{center}
  \pgfplotstableread{tabledata/Rinv_robustness_minres.dat}       \rinv
\pgfplotstableread{tabledata/lam_robustness_minres.dat}       \lam
\pgfplotstableread{tabledata/S_robustness_minres.dat}       \S
                         
\begin{tikzpicture}
  [
  %spy using outlines={rounded rectangle, width=1.5cm, height=0.15cm, very thick,black,magnification=3, connect spies}, 
  scale=0.65
  ]
  \begin{groupplot}[
    group style ={group size = 3 by 1,horizontal sep=30pt,vertical sep=60pt},
    name=plot2,
    scale=1.0,
    legend columns=4,
    %legend rows=3,
    legend style={text width=3em,text height=0.9em },
    %ylabel=it,
    % ylabel=$U_B$,
    %xmin = 2e2,
    %xmax=2e5,
    xmode=log,
    %ymax=1e-1,
    %ymin=1e-5,
    %ymode=log,
    %ytick = {1e-5,1e-4,1e-3,1e-2,1e-1},
    %y tick label style={/pgf/number format/.cd,fixed,precision=2},
    %x tick label style={/pgf/number format/.cd,fixed,precision=2},
    yticklabel style={text width=3em,align=right},
    xticklabel style={text width=3em,align=left},
    % 
    % log basis x=2,
     grid=major,
    %% major grid style={black!2},
    %legend style={
    %  cells={align=left},
    %  % draw=none,
    %  at={(0.1,0.1)},
    %  anchor = north west
    %},
    % xticklabel=\pgfmathparse{2^\tick}\pgfmathprintnumber{\pgfmathresult}
    ]
        
    \nextgroupplot[legend to name={mylegend}, xlabel=$R^{-1}$, xtick = {1,1e2,1e4,1e6,1e8,1e10,1e12}]
    \addlegendentry{$l=1$}
    \addlegendimage{black}
    \addlegendentry{$l=2$}
    \addlegendimage{red}
    \addlegendentry{$l=3$}
    \addlegendimage{green!70!blue}
    \addlegendentry{$l=4$}
    \addlegendimage{red!50!yellow}
    
    \addplot[line width=1pt, mark=x, mark options={solid}, color=black] table[x=0,y=1]{\rinv};
    \addplot[line width=1pt, mark=x, mark options={solid}, color=red] table[x=0,y=2]{\rinv};
    \addplot[line width=1pt, mark=x, mark options={solid}, color=green!70!blue] table[x=0,y=3]{\rinv};
    \addplot[line width=1pt, mark=x, mark options={solid}, color=red!50!yellow] table[x=0,y=4]{\rinv};

    \nextgroupplot[legend to name=dummy, xlabel = $\lambda$, xtick = {1,1e1,1e2,1e3,1e4,1e5,1e6}]
    \addplot[line width=1pt,mark=x, mark options={solid}, color=black] table[x=0,y=1]{\lam};
    \addplot[line width=1pt,mark=x, mark options={solid}, color=red] table[x=0,y=2]{\lam};
    \addplot[line width=1pt,mark=x, mark options={solid}, color=green!70!blue] table[x=0,y=3]{\lam};
    \addplot[line width=1pt,mark=x, mark options={solid}, color=red!50!yellow] table[x=0,y=4]{\lam};

    \nextgroupplot[legend to name=dummy, xlabel = $S$, xtick = {1e-16,1e-12,1e-8,1e-4,1}, ytick = {9,10}, ymin = 8.5,ymax = 10.5]
    \addplot[line width=1pt, mark=x, mark options={solid}, color=black] table[x=0,y=1]{\S};
    \addplot[line width=1pt, mark=x, mark options={solid}, color=red] table[x=0,y=2]{\S};
    \addplot[line width=1pt, mark=x, mark options={solid}, color=green!70!blue] table[x=0,y=3]{\S};
    \addplot[line width=1pt, mark=x, mark options={solid}, color=red!50!yellow] table[x=0,y=4]{\S};
       
  \end{groupplot}

  \node[right=1em,inner sep=0pt] at (5,-1.7) {\pgfplotslegendfromname{mylegend}};

\end{tikzpicture}
 
%%% Local Variables:
%%% mode: latex
%%% TeX-master:"../main"
%%% End:
  \vspace{5pt}
  \caption{Robustness of the preconditioner defined in \eqref{eq::precondtwo}} \label{fig:exone}
  \end{center}
\end{figure}

\subsection{Cost-efficiency of the new family of hybridized discretizations}
%Finally, we demonstrate the cost-efficiency of the new families of discretizations.
\subsubsection{DG vs HDG}
In a first step we only illustrate the effect of hybridization
introduced Section~\ref{sec::HDG}. To this end we consider the model
problem: Find $\bu \in \bm{H}^1_0(\Omega)$ such that
\begin{align*}
  - \div(\eps(\bu)) = \bf,
\end{align*}
with a given right hand side $\bf$ and $\Omega = (0,1)^3$. We solve this problem on a given triangulation with 166 elements either with an $H(\div)$-conforming DG or HDG method, i.e. setting $\lambda = 0$ we have the problems:
Find $\bu_h \in \bU_h$ such that
\begin{align}\label{eq::veclapDG}
  a_h^{\text{DG}}(\bu_h, \bv_h) = (f,\bv_h) \quad \forall \bv_h \in \bu_h,
\end{align}
and find $(\bu_h, \hat {\bu}_h) \in \overline{\bU}_h$ such that
\begin{align}\label{eq::veclapHDG}
  a_h^{\text{HDG}}((\bu_h, \hat{\bu}_h), (\bv_h, \hat{\bv}_h)) = (f,\bv_h), \quad \forall (\bv_h, \hat {\bv}_h) \in \overline{\bU}_h.
\end{align}
In Table~\ref{tab::dofsfordgandhdg} we compare the values
\begin{enumerate}
\item[] \texttt{dof}: number of unknowns,
\item[] \texttt{cdof}: number of coupling unknowns,
\item[] \texttt{nze}: number of non-zero entries in thousands of the resulting system matrix,
\end{enumerate}
for varying polynomial degrees $l=1,\ldots,6$ which correspond to the
local order of approximation of $\bm u_h, \hat {\bu}_h$ in
$\text{BDM}_{\ell}(T)/\text{P}^\perp_{\ell}(F)$, for all
$T \in \mathcal{T}_h$ and all $F \in \mathcal{F}_h$. Here
$\text{P}^\perp_{\ell}(F)$ is the space of polynomials of order $l$
that are orthogonal to the normal vector, see the definition of the
space $\widehat {\bm U}_h$ in Section~\ref{sec::HDG}.

First, note that, due to the coupling between element unknowns in the
DG method, no static condensation can be applied, i.e. \texttt{dof} =
\texttt{cdof}.  When solving the linear system one is
particularly interested in the number of non-zero entries. As we can
see, the HDG method clearly outperforms the DG method in case of
higher order approximation ($l \ge 4$). In the low order cases, the
additional unknowns introduced by the new facet unknowns dominate, and
thus no improvement can be expected.

\begin{remark} \label{remark::PHDG}
  The HDG method can further be improved by means of another
  technique, called ``projected jumps'', which was introduced
  in~\cite{LehrenfeldSchoeberl2016high}. This modification allows to
  further decrease the coupling of the HDG method without affecting
  its approximation properties. This essentially compensates the
  overhead of the HDG method in the low order cases by reducing the
  polynomial degree of the space of $\hat {\bu}_h$ to
  $\text{P}^\perp_{\ell-1}(F)$ and adding consistent projections in
  the bilinear form. Although we do not discuss these modifications
  here, we include the corresponding numbers in Table
  \ref{tab::dofsfordgandhdg} in the rows denoted by PHDG.  Note that
  the well-posedness theory and the robustness of the preconditioners
  obtained in this work also hold for the PHDG method.
\end{remark}

\begin{table}[h]
  \scriptsize
  \centering
  \begin{tabular}{@{~}c@{~}|@{~~~~~~}c@{~~~}c@{~~~}c|c@{~~~}c@{~~~}c@{~~~~~~}|c@{~~~}c@{~~~}c@{~~~~~~}}
    %    \toprule
    &  \texttt{dof} &  \texttt{cdof} &  \texttt{nze} & \texttt{dof} &  \texttt{cdof} &  \texttt{nze}& \texttt{dof} &  \texttt{cdof} &  \texttt{nze}\\
    \midrule
    &\multicolumn{3}{c|}{$l=1$} & \multicolumn{3}{c|}{$l=2$}& \multicolumn{3}{c}{$l=3$}\\
    DG & 834 & 834 & 65 & 2664 & 2664 & 454 & 6100 & 6100 & 1945\\
    HDG & 2502 & 2502 & 193 & 6000 & 5004 & 770 & 11660 & 8340 & 2140\\
    PHDG & 1390 & 1390 & 59 & 4332 & 3336 & 342 & 9436 & 6116 & 1151\\
    \midrule
    &\multicolumn{3}{c|}{$l=4$} & \multicolumn{3}{c}{$l=5$}& \multicolumn{3}{c}{$l=6$}\\
    DG & 11640 & 11640 & 6238 & 19782 & 19782 & 16512 & 31024 & 31024 & 38106\\
    HDG & 19980 & 12510 & 4815 & 31458 & 17514 & 9438 & 46592 & 23352 & 16779 \\
    PHDG & 17200 & 9730 & 2913 & 28122 & 14178 & 6185 & 42700 & 19460 & 11652
  \end{tabular} 
  \caption{\texttt{dof}, \texttt{cdof} and \texttt{nze} of the system
    matrix of the DG, HDG and PHDG methods for different polynomial
    degrees $l$.}
  \label{tab::dofsfordgandhdg}
\end{table}

\subsubsection{Mixed vs hybrid mixed methods}

We continue discussing the modifications introduced in
Section~\ref{sec::hybridizeddarcy} with regards to the following Darcy
model problem: Find
$(\bw,p) \in \bm{H}(\div)(\Omega) \times L^2(\Omega)$ such that
\begin{align*}
  \bw + \nabla p &= 0, \\
  \div(\bw)  &= g,
\end{align*}
for a given right hand side $g$ on the domain $\Omega = (0,1)^3$. We
use the same mesh as in the previous section, and consider the
problems: Find $(\bw_h, p_h ) \in \bm W_h \times P_h$, such that
\begin{subequations}  \label{eq::darcymixed}
\begin{alignat}{2}
  (\bw_{h},\bz_{h}) {-} (p_{h},\divv \bz_{h}) &= 0,& \qquad& \forall \bz_h\in  \bm W_h, \\
  -(\divv\bw_{h},q_{h})  &= -(g,q_{h}), &\qquad& \forall q_h \in P_h.
\end{alignat}  
\end{subequations}
and find $(\bw_h, (p_h,\hat p_h) ) \in \bm W^-_h \times \overline P_h$, such that
\begin{subequations}\label{eq::darcyhybridmixed}
\begin{alignat}{2}
  (\bw_{h},\bz_{h}) {-} b((p_{h}, \hat p_h), \bz_{h}) &= 0, &\quad& \forall \bz_h\in  \bm W^-_h, \\
-  b((q_{h},\hat q_h),\bw_{h}) &= -(g,q_{h}), &\quad& \forall(q_h, \hat q_h)) \in\overline P_h.
\end{alignat}
\end{subequations}
Note that equation \eqref{eq::darcymixed} only allows a static
condensation of the following degrees of freedom: all local
(element-associated) degrees of freedom of the space $\bm W_h$,
i.e. element-wise basis functions with a vanishing normal trace, and all
high-order (considering a standard $L^2$-Dubiner basis) basis
functions of $P_h$ such that element-wise constant basis functions
remain in the system.  In contrast to this, system
\eqref{eq::darcyhybridmixed} allows us to eliminate all degrees of
freedom associated with the basis functions of the spaces $\bm W_h^-$
and $P_h$.  In Table~\ref{tab::dofsformandhm}, we again present the
corresponding numbers as discussed above, where M represents the
discretization of \eqref{eq::darcymixed}, and HM of
\eqref{eq::darcyhybridmixed}.  Here, the order $l$ corresponds to the
local approximation polynomial degree of $\bm w_h, p_h, \hat p_h$ in
$\text{RT}_{\ell}(T)/\text{P}_{\ell}(T)/\text{P}_{\ell}(F)$, for all
$T \in \mathcal{T}_h$ and all $F \in \mathcal{F}_h$.  
Further, we observe that the hybrid
mixed method produces always a smaller number of non-zero entries than the standard mixed method although 
the difference is negligible. 
%Similarly, as
%before, when considering the number of non-zero entries, the hybrid
%mixed method is always better than the standard mixed method although
%the improvement in negligible.  
However, 
%keep in mind that 
the main purposes of hybridization are a reduction of the number of
coupling dof and obtaining a condensed system,
cf.~\eqref{eq::precondtwo}, which is symmetric positive definite, see
also~\cite{CockburnEtAl2009unified, MR2051067}. The latter allows us
to use preconditioners for $H^1$-elliptic problems like standard
algebraic multigrid methods.

\begin{table}
  \scriptsize
  \centering
  \begin{tabular}[]{@{~}c@{~}|@{~~~~~~}c@{~~~}c@{~~~}c|c@{~~~}c@{~~~}c@{~~~~~~}|c@{~~~}c@{~~~}c@{~~~~~~}}
    %    \toprule
    &  \texttt{dof} &  \texttt{cdof} &  \texttt{nze} & \texttt{dof} &  \texttt{cdof} &  \texttt{nze}& \texttt{dof} &  \texttt{cdof} &  \texttt{nze}\\
    \midrule
    &\multicolumn{3}{c|}{$l=0$} & \multicolumn{3}{c|}{$l=1$}& \multicolumn{3}{c}{$l=2$}\\
    M & 552 & 552 & 4k & 2320 & 1324 & 26k & 5968 & 2482 & 94k\\
    HM & 1108 & 278 & 2k & 3988 & 834 & 21k & 9304 & 1668 & 86k\\
    \midrule
    &\multicolumn{3}{c|}{$l=3$} & \multicolumn{3}{c}{$l=4$}& \multicolumn{3}{c}{$l=5$}\\
    M & 12160 & 4026 & 251k & 21560 & 5956 & 555k & 34832 & 8272 & 1077k\\
    HM & 17720 & 2780 & 238k & 29900 & 4170 & 535k & 46508 & 5838 & 1049k
  \end{tabular}
  \caption{\texttt{dof}, \texttt{cdof} and \texttt{nze} of the system
    matrix for different polynomial degrees $l$ in the discretizations
    of problems \eqref{eq::darcymixed}, 
    \eqref{eq::darcyhybridmixed}.}
   \label{tab::dofsformandhm}
 \end{table}

\section{Acknowledgement}

The second and the last author acknowledge the support by the Austrian
Science Fund (FWF) through the research programm ``Taming complexity
in partial differential systems'' (F65) - project ``Automated
discretization in multiphysics'' (P10).

\bibliographystyle{plain}
\bibliography{HDG_Biot}

\end{document}